\documentclass[12pt]{article}
\usepackage{amsmath,amssymb,amsfonts,theorem}
\usepackage{graphicx,color}
\usepackage{subfigure}
\graphicspath{{./Figures/}}

\DeclareMathOperator*{\bigcart}{\raisebox{-0.3ex}{\text{\Large$\times$}}}
{ \theorembodyfont{\normalfont} 

\newtheorem{remark}{Remark}
}
\newtheorem{assumption}{Assumption}

\newtheorem{theorem}{Theorem}
\newtheorem{lemma}{Lemma}

\newtheorem{proposition}{Proposition}

\newcommand{\bigintersect}{\bigcap}

\newcommand{\N}{\mathbb{N}}

\newcommand{\R}{\mathbb{R}}

\newcommand{\qd}{\texttt{q}}    

\def\stopmodifc{\color{black}} 
\newcommand{\bluff}{{\hbox{\raise 15pt \hbox{\hskip 0.5pt}}}}
\newcommand{\rfb}[1]{\mbox{\rm
   (\ref{#1})}\ifx\undefined\stillediting\else:\fbox{$#1$}\fi}

\newcommand{\dd}   {{\rm d}\hbox{\hskip 0.5pt}}

\newcommand{\dst}{\displaystyle}

\def\be{\begin{equation}}
\def\ee{\end{equation}}
\def\ba{\begin{array}}
\def\ea{\end{array}}
\def\eqa{\begin{eqnarray}}
\def\eqe{\end{eqnarray}}
\title{\LARGE \bf Coordination of passive systems under quantized measurements
\thanks{An abridged 
version of this paper has been
presented
at the 50th IEEE Conference on Decision and Control and European Control
Conference, December 12-15, 2011, Orlando, FL. 
This work is partially supported by 
the Dutch Organization
for Scientific Research (NWO) within the project {\it QUantized Information Control for formation Keeping} (QUICK).
}
}

\author{
Claudio De Persis \footnote{ITM, Faculty of Mathematics and
Natural Sciences, University of Groningen, the
Netherlands, Tel: +31 50 363 3080, Email:
{\tt c.de.persis@rug.nl},
 and Dipartimento
di Informatica e Sistemistica, Sapienza Universit\`a di Roma,  Via
Ariosto 25, 00185 Roma, Italy.}
\and
Bayu Jayawardhana \footnote{
ITM, Faculty of Mathematics and
Natural Sciences, University of Groningen, the
Netherlands, Tel: +31 50 363 7156, Email:
{\tt bayujw@ieee.org, b.jayawardhana@rug.nl}}
}

\begin{document}

\maketitle

\begin{abstract}
In this paper we investigate a passivity approach to collective
coordination and synchronization problems in the presence of quantized measurements
and show that coordination tasks can be achieved in a practical
sense for a large class of passive systems.
 \end{abstract}


\section{Introduction}
In the very active area of consensus, synchronization and
coordinated control there has been an increasing interest in the
use of quantized measurements and control
(\cite{kashyap.et.al.aut07, AN-AO-AO-JNT:09, frasca.et.al.ijrnc09,
carli.et.al.ijrnc10, li.et.al.tac11, censi.murray.acc09} and
references therein). As a matter of fact, since these problems
investigate systems or agents which are distributed over a
network, it is very likely that the agents must exchange
information over a digital communication channel and quantization
is one of the basic limitations induced by finite bandwidth
channels. To cope with this limitation, measurements are processed
by quantizers, i.e.\ discontinuous maps taking values in a
discrete or finite set.
Another reason to consider quantized measurements stems
from the use of coarse sensors.\\
The use of quantized measurements induces a partition of the space
of measurements: whenever the measurement function crosses the
boundary between two adjacent sets of the partition, a new value
is broadcast through the channel. As a consequence, when the
networked system under consideration evolves in continuous time,
as it is often the case with e.g.\ problems of coordinated motion,
the use of quantized measurements results in a completely
asynchronous exchange of information among the agents of the
network. Despite the asynchronous information exchange and the use
of a discrete set of information values, meaningful examples of
synchronization or coordination can be obtained
(\cite{dimarogonas.johansson.aut10,
ceragioli.et.al.aut11, liu.cao.ifac11,cortes.aut06}).\\
There are other approaches to reach synchronization (consensus) with an asynchronous exchange of information, such as gossiping (\cite{boyd.et.al.tif06,carli.et.al.aut10}), where at each time step two randomly chosen agents exchange information. The latter approach is conceptually very different from the one considered in this paper, where the agents broadcast information when a local event occurs (the measurement crosses the partition boundary). Moreover, while gossiping algorithms are mainly devised for discrete-time systems, here we focus on continuous-time systems.
\\
In view of the several contributions to quantized coordination
problems available for discrete-time systems
(\cite{kashyap.et.al.aut07, AN-AO-AO-JNT:09, frasca.et.al.ijrnc09,
carli.et.al.ijrnc10, li.et.al.tac11, censi.murray.acc09}), one may
wonder whether it would be more convenient simply to derive the
sampled-data model of the system and then apply the discrete-time
results. Due to the distributed nature of the system, a
sampled-data approach to the design of coordinated motion
algorithms presents a few drawbacks: it might require synchronous
sampling at all the nodes of the network and consequent accurate
synchronization of all the node clocks; it might also require fast
sampling rates, which may not be feasible in a networked system
with a large number of nodes and connections. Finally, the
sampled-data model may not fully preserve some of the features of
the original model.
For these reasons, we focus here on continuous-time coordination
problems under quantized measurements.\\
A few  works on this class of problems have recently
appeared. The work \cite{cortes.aut06} deals with consensus
algorithms using binary control algorithms. In
\cite{dimarogonas.johansson.aut10} the attention is turned to
quantized measurements and the consensus problem under quantized
{\em relative} measurements is tackled. The same problem, but
considering quantized {\it absolute} measurements, is studied in
\cite{ceragioli.et.al.aut11}. The paper also introduces hysteretic
quantizers to prevent the occurrence of chattering due to the
presence of sliding modes. More recently, the work
\cite{liu.cao.ifac11} has studied the quantized consensus
algorithm for double integrators. A remarkable advancement in the
study of consensus algorithms over {\em time-varying}
communication graphs and using quantized measurements has been
provided by \cite{frasca.arxiv11}.
\\
Despite the unquestionable interest of the results in papers such
as (\cite{dimarogonas.johansson.aut10, ceragioli.et.al.aut11,
liu.cao.ifac11, cortes.aut06,frasca.arxiv11}), they present an
important limitation: they focus on agents with simple dynamics
such as single (\cite{dimarogonas.johansson.aut10,
ceragioli.et.al.aut11, cortes.aut06,frasca.arxiv11}) or double
integrators (\cite{liu.cao.ifac11}).
The goal of this paper is to investigate the potentials of an
approach to coordinated motion and synchronization which takes into account
simultaneously complex dynamics for the agents of the network and
quantized measurements.\\
In coordinated
motion, variables of interest are the position and the velocity
of each subsystem, and the problem is to devise control laws which
guarantee prescribed inter-agent positions and velocity tracking.
In this paper we focus on the approach to coordinated motion
proposed in \cite{arcak.tac07}. In
that paper, the author has shown how a number of coordination
tasks could be achieved for a class of passive nonlinear systems
and has been using this approach for related problems in
subsequent work (\cite{bai.et.al.aut08, bai.et.al.aut09}).
Others
have been exploiting passivity
(\cite{chopra.spong.cdc06,hokayem.et.al.ijrnc09,schaft.maschke.necsys10,nuno.et.al.tac11}
to name a few) in connection with coordination problems.
Our interest for the approach in \cite{arcak.tac07} stems from the fact
that it allows to deal with complex coordination tasks, including
consensus with velocity tracking, in way that
naturally lends itself to deal with the presence of quantized
measurements. In the approach of \cite{arcak.tac07},
a continuous feedback law is designed
to achieve the desired coordination task under appropriate
conditions. Thus the presence of quantized measurements can be
taken into account in this setting by introducing in the feedback
law static discontinuous maps (the previously recalled
quantizers). Although in the case of quantized measurements
the conditions in \cite{arcak.tac07}
are not fulfilled due to the discontinuous nature of the quantizers,
one can argue that an approximate or ``practical"
(\cite{ceragioli.et.al.aut11}) coordination task is achievable
under suitably modified conditions. This is the idea which is
pursued in this paper. In the case of a control system with a
single communication channel this was studied in
\cite{ceragioli.depersis.scl07}. Another reason to
consider the approach of \cite{arcak.tac07} is that it provides a
systematic way to deal with a large variety of cooperative control
problems, as it has been authoritatively proven in the recent book
\cite{bai.et.al.book}.
\\
A second aim of this paper is to study practical state
synchronization under quantized output feedback. In these
problems, one investigates conditions under which the state
variables of all the subsystems asymptotically converge to each
other, with no additional requirement on the velocity tracking.
Passivity
(\cite{chopra.spong.cdc06,stan.sepulchre.tac07,scardovi.sepulchre.aut}),
or the weaker notion of semi-passivity
(\cite{pogromsky.96,pogromsky.nijmeijer.csi01,
nijmeijer.steur.cs11}), has also played an important role in
synchronization problems. Here we mainly focus on the models
considered in \cite{chopra.spong.cdc06,scardovi.sepulchre.aut}.
\\
The main contribution of this paper is to show that
some of the results of  \cite{arcak.tac07} and \cite{scardovi.sepulchre.aut}
hold in a practical sense in the presence of quantized measurements. Because
the latter introduces discontinuities in the system, a rigorous
analysis is carried out relying on notions and tools from
nonsmooth control theory and differential inclusions.
As far as the coordination problem is concerned, although the
passivity approach of \cite{arcak.tac07} allows to consider a large variety of
coordination control problems, in this paper
we mainly focus on agreement problems in which agents aim at
converging to the same position.
\\
A few other papers have appeared which deal with coordination problems for
passive systems in the presence of quantization.
The work \cite{hokayem.et.al.ijrnc09} deals with a position coordination problem for
Lagrangian systems when delays and limited data rates are affecting the system.
The paper \cite{fradkov.et.al.cs1.09} deals with master-slave synchronization of
passifiable Lurie systems when the master and the slave
communicate over a limited data rate channel. The main difference of our paper compared with
\cite{hokayem.et.al.ijrnc09,fradkov.et.al.cs1.09} is that in the former
each system in the network transmits quantized information in a completely
asynchronous fashion and no common sampling time is required. From a mathematical
point of view, this means that our approach yields a discontinuous
closed-loop system as opposed to a sampled-data one. Moreover, the classes of systems and
the coordination problems considered here appear to be different
from those in \cite{hokayem.et.al.ijrnc09,fradkov.et.al.cs1.09}.
\\
The organization of the paper is as follows. The passivity
approach to coordination problems is recalled in Section
\ref{sect.passivity.approach}. In Section \ref{quant.coord.contr}
the coordination control problem
in the presence of uniform quantizers 
is formulated and the main results
are presented along with some examples.
The synchronization problem for passive systems under
quantized output feedback is studied in Section \ref{quant.synchro}.
In Section \ref{sec.final} a few guidelines for future research are
discussed.
In the Appendix some technical tools are reviewed for
the sake of readers' convenience.


\section{Preliminaries}
\label{sect.passivity.approach}

\vspace{-0.2cm}
Consider $N$ systems connected over an undirected graph $G=(V,E)$,
where $V$ is a set of $N$ nodes and $E\subseteq V\times V$ is a
set of $M$ edges connecting the nodes.
The standing assumption throughout the paper is that the graph $G$ is
{\em connected}.
Each system $i$, with
$i=1,2,\ldots, N$, is associated to the node $i$ of the graph and
the edges connect the nodes or systems which communicate. \\
Each system $i$ is described by
\begin{equation}\label{feedb.path}
\Sigma_i : \begin{array}{rl}
\dot \xi_i & = f_i(\xi_i) + g_i(\xi_i)u_i \\
w_i & = h_i(\xi_i) + v_i,\end{array}
\end{equation}
where the state $\xi_i\in \R^{n_i}$, the input $u_i\in \R^p$, the output $w_i\in \R^p$, the exogenous signal $v_i\in \R^p$ and
the maps
$f_i, g_i, h_i$ are assumed to be locally Lipschitz satisfying $f_i(\mathbf{0})=\mathbf{0}$,
$g_i(\mathbf{0})$ full column-rank,
$h_i(\mathbf{0})=\mathbf{0}$.
For the system $\Sigma_i$, we assume the following:

\begin{assumption}\label{assum.passivity}
There exists a
continuously differentiable
storage function $S_i:\R^{n_i}\to\R_+$ 
which is positive definite
and radially unbounded such that
\begin{equation}\label{passivity}
\nabla S_i(\xi_i)(\bluff f_i(\xi_i) + g_i(\xi_i)u_i )\leq -W_i(\xi_i) + h_i(\xi_i)^Tu_i,
\end{equation}
where $W_i$ is a
continous positive 
function which is zero at the origin.
\end{assumption}
Such a system $\Sigma_i$ is called a strictly-passive
system (with $v_i=0$). If $W_i$ is a 
non-negative
function, then $\Sigma_i$  is called a passive system.
\\
Label one end of each edge in $E$ by a positive sign and the other
one by a negative sign. Now, consider the $k$-th edge in $E$, with
$k\in \{1,2, \ldots,M\}$, and let $i,j$ be the two nodes connected
by the edge.
For the coordination problem, which is detailed in Subsection \ref{subsectioncoordination},
the relative measurements of the integral form $\int_0^t w_i(\tau)\dd \tau$ and $\int_0^t w_j(\tau)\dd\tau$ are used.
On the other hand, for the
synchronization problem, which is briefly reviewed in Subsection \ref{subsectionsynchronization},
we need the relative measurements of the signals $w_i$ and $w_j$.
Thus, depending upon specific problems, let $z_k$ describe the
difference between the signals $w_i$ and $w_j$
(or the difference between the signals
$x_i(t):=\int_0^t w_i(\tau)\dd \tau + x_i(0)$ and $x_j(t):=\int_0^t w_j(\tau)\dd\tau + x_j(0)$
with constant vectors $x_i(0), x_j(0) \in\R^{p}$)
and be defined as follows:
\[
z_k=\left\{\ba{ll}
w_i-w_j \;\; \text{(or} \;\; x_i-x_j\;\text{)} & \mbox{if $i$ is the positive end of the edge $k$}\\
w_j-w_i \;\; \text{(or} \;\; x_j-x_i\;\text{)} & \mbox{if $i$ is the negative end of the edge $k$}\;.
\ea\right.
\]
Recall also that the incidence matrix $D$ associated with the
graph $G$ is the $N\times M$ matrix such that
\[
d_{ik}=\left\{ \ba{rl}
+1 &\mbox{if node $i$ is the positive end of edge $k$}\\
-1 &\mbox{if node $i$ is the negative end of edge $k$}\\
0 & \mbox{otherwise}. \ea \right.
\]
By the definition of $D$, the variables $z$ can be concisely
represented as
\be\label{z} z=(D^T\otimes I_p) w\; \; \left(\bluff \text{or} \;\; z=(D^T\otimes I_p) x \; \right) \ee
where $w=[w_1^T\ldots w_N^T]^T$ and $x=[x_1^T\ldots x_N^T]^T$, respectively,
and the
symbol $\otimes$ denotes the Kronecker product of matrices (see
Appendix \ref{a-1} for a definition). \\
%
In this paper we are interested in control laws which use
quantized measurements. For each $k=1,2,\ldots, M$, instead of
$z_k$, the vector
\[
\qd(z_k):=(\qd(z_{k1})\ldots \qd(z_{kp}))^T
\]
is available, where $\qd$ is the quantizer map which is defined as follows.
Given a positive real number $\Delta$, we let $\qd: \R\to
\mathbb{Z}\Delta$ be the function
\be\label{scalar.quantizer} \qd(r)=
\Delta\left\lfloor \frac{r}{\Delta}+\frac{1}{2}\right\rfloor
\ee
with $\frac{1}{\Delta}$ the precision of the quantizer. As
$\Delta\to 0$, $\qd(r)\to r$. 
Observe that each entry of $z_k$ is quantized independently of the
others and the quantized information is then used in the control
law.
\begin{remark}
The results of the paper continue to hold if each
 quantizer  has its own resolution (that is, the information $z_{kj}$ is quantized by a quantizer
 with resolution $\Delta_{kj}$). However, to reduce the notational burden, we only deal with the case in which the quantizers have all the same resolution $\Delta$.
\end{remark}
\stopmodifc

In the following subsections, we review the results on passivity approach to the coordination
problems of \cite{arcak.tac07} and to the synchronization problems of \cite{scardovi.sepulchre.aut} without the quantized measurements.

\subsection{Passivity approach to the coordination problem}\label{subsectioncoordination}

In the coordination problems of \cite{arcak.tac07}, the signal $w_i$ of each system $\Sigma_i$ corresponds to
the velocity of the system, and thus, $x_i$, $i=1,\ldots,N$, represents the positions which must be coordinated
(recall that $x_i(t):=\int_0^t w_i(\tau)\dd \tau + x_i(0)$).
The coordination problem under consideration requires all the systems of the formation to move with a prescribed velocity
$v$, i.e., $v_1=v_2=\ldots=v_N=v$. Define
\be\label{feedb.path0}
y_i=\dot x_i-v
\ee
the velocity tracking
error. It can be checked from (\ref{feedb.path}) and the definition of $\dot x_i$ that $y_i=h(\xi_i)$.
The standing assumption is that, possibly after
a preliminary feedback which uses information available locally,
each system $\Sigma_i$ is strictly passive, i.e., (\ref{passivity}) holds
with $W_i$ positive definite. In other words,
it is strictly passive from the control input
$u_i$ to the velocity error $y_i$.
%
%
%
\\
For the sake of conciseness, the equations
(\ref{feedb.path}), (\ref{feedb.path0}) are
rewritten as
\be\label{closed.loop.x.xi}
\ba{rcl}
\dot x&=& \underbrace{\left(\ba{c} h_1(\xi_1)\\ \vdots \\
h_N(\xi_N) \ea\right)}_{h(\xi)} + \underbrace{ \left(\ba{c} v\\
\vdots \\ v \ea\right)}_{\mathbf{1}_N\otimes v}
\\
\dot \xi&=&
\underbrace{ \left(\ba{c} f_1(\xi_1)\\ \vdots \\ f_N(\xi_N)
\ea\right)}_{f(\xi)} + \underbrace{ \left(\ba{ccc}
g_1(\xi_1)&  \ldots &  \mathbf{0}\\
\vdots &   \ddots & \vdots\\
\mathbf{0}&  \ldots & g_N(\xi_N)\\
\ea\right)}_{g(\xi)}u
\ea\ee
where $x=(x_1^T\ldots x_N^T)^T$, $\xi=[\xi_1^T\ldots\xi_N^T]^T$, $u=[u_1^T\ldots
u_N^T]^T$, $\mathbf{1}_N$ is the $N$-dimensional vector whose
entries are all equal to $1$ and $\mathbf{0}$ denotes a vector
 of appropriate dimension of all zeros.

The formation control problem consists of designing each control
law  $u_i$, with $i=1,2,\ldots,N$, in such a way that it uses only
the information available to the agent $i$ and guarantees the
following two specifications:
\begin{description}
\item{(i)} $\lim_{t\to\infty} |\dot x_i(t)-v(t)| =0$ for each $i=1,2,\ldots, N$,
with $v(t)$ a bounded and piece-wise continuous reference velocity
for the formation;
\item{(ii)} $z_k(t)\to {\cal A}_k$ as $t\to \infty$ for each
$k=1,2,\ldots,M$, where ${\cal A}_k\subset \R^p$ are the
prescribed sets of convergence \footnote{We refer the interested
reader to \cite{arcak.tac07} for examples of sets
 ${\cal A}_k$ related to some coordination problems.
 The sets ${\cal A}_k$  which are of interest in this paper will be introduced in
 (\ref{A.k}).} and $z=(D^T\otimes I_p) x$ as defined in (\ref{z}).
\end{description}
 In \cite{arcak.tac07}, where measurements without quantization are considered,
the case 
${\cal A}_k=\{\mathbf{0}\}$
is referred to as the {\em agreement problem}.

Let $P_k:\R^p\to \R$, for $k=1,2,\ldots, N$, be nonnegative continuously differentiable
(the latter assumption will be removed in the next section)
and radially unbounded functions whose minimum
is achieved at the points in ${\cal A}_k$. To be more precise,
the functions $P_k$ are assumed to satisfy
\be\label{properties.P}
P_k(z_k)=0\;\mbox{and}\; \nabla
P_k(z_k)=\mathbf{0}\quad \mbox{if and only if}\quad  z_k\in {\cal
A}_k. \ee
Define
\be\label{psi.via.P} \nabla P_k(z_k)=\psi_k(z_k)\;. \ee
The feedback laws proposed in  \cite{arcak.tac07} to solve the
problem formulated above are:
\be\label{arcak.control} u_i=-\dst\sum_{k=1}^M d_{ik}
\psi_k(z_k)\;,\; i=1,2,\ldots,N. \ee
Observe that, as required, each control law $u_i$ uses only information
which is available to the agent $i$. Indeed, $d_{ik}\ne 0$ if and
only if the edge $k$ connects $i$ to one of its neighbors. In
compact form, (\ref{arcak.control}) can be rewritten as
\be\label{arcak.control.compact} u=-(D\otimes I_p)\psi(z)\;, \ee
where
$\psi(z)=[\psi_1(z_1)^T\ldots\psi_M(z_M)^T]^T$
and $z$ is as in
(\ref{z}).
Before ending the section, we recall
that the system below with input $\dot x$ and output $-u$, namely
(see Figure 2 in \cite{arcak.tac07} for a pictorial representation
of the system)
\be\label{feedf} \ba{rcl}
\dot z&=& (D^T\otimes I_p) \dot x\\
-u &=& (D\otimes I_p)\psi(z) \ea\ee
is passive from $\dot x$ to $-u$
with storage function $\sum_{k=1}^M P_k(z_k)$.
We remark that the function $P_k(z_k)$ is chosen in such a way that the
region where the variable $z_k$ must converge for the system to
achieve the prescribed coordination task coincides with the set of
the global minima of $P_k(z_k)$. Hence, the coordination task
guides the design of $P_k(z_k)$ which in turn allows to determine
the control functions (\ref{arcak.control}) via (\ref{psi.via.P}).
The functions $P_k(z_k)$ in the case of agreement problems via
quantized control laws will be designed in Section \ref{quant.coord.contr}.

\subsection{Passivity approach to the synchronization problem}\label{subsectionsynchronization}

In the synchronization problem of
\cite[Theorem 4]{scardovi.sepulchre.aut},
each system $\Sigma_i$ in (\ref{feedb.path}) (with $v_i=0$) is assumed to be linear, identical and passive.
For such setting, each (passive) system $\Sigma_i$ is of the form
\be\label{lti}\ba{rcll}
\dot \xi_i &=&  A \xi_i +B u_i\\
w_i &=& C \xi_i & i=1,2,\ldots, N \ea\ee
where $\xi_i\in \R^n$, $u_i,
w_i\in \R^p$ and the
passivity of $\Sigma_i$ implies that the
following assumption holds:
\begin{assumption}\label{a.passivity}
There exists an $(n\times n)$ matrix $P=P^T>0$ such
that
\[
A^T P+PA\le 0,\; B^T P=C\;.
\]
\end{assumption}

The synchronization problems can then be stated as designing each control
law $u_i$, $i=1,2,\ldots,N$, using only the information available to te agent $i$ such that,
for every $i$, $\xi_i-\xi_0 \to {\cal A}$
where $\xi_0$ is the trajectory of the autonomous system $\dot \xi_0 = A \xi_0$ which is initialized by the
average of the initial states, i.e., $\xi_0(0)=\frac{1}{N}\sum_i\xi_i(0)$, and ${\cal A}\subset \R^p$ is the
prescribed set of convergence. In the case without the quantized measurements, which is treated in \cite{scardovi.sepulchre.aut},
${\cal A}=\{\mathbf{0}\}$.
%
%
The coordination problem that is reviewed in Subsection \ref{subsectioncoordination}, is related to the
case when $\dot\xi_0=\mathbf{0}$ \cite{arcak.tac07}.
For another viewpoint, we can consider that (\ref{lti}) corresponds to the case in the Subsection \ref{subsectioncoordination}, where the
mapping $u\mapsto y$ is an identity operator, $v=\mathbf{0}$ and one takes into
account dynamics on the subsystem $x$ which are more complex than those of a single integrator.

In addition to output synchronization, it is well-known that
the {\em states} of interconnected passive systems synchronize under observability assumption
(\cite{chopra.spong.cdc06}). The largest invariant set of the interconnected systems,
when the measurements are not quantized and $(C,A)$ is
observable, is the set $\{\xi\in \R^{nN}:\xi_1=\ldots=\xi_N\}$. In the case of
quantized measurements, the
invariant set is larger.
Our main result in Section \ref{quant.synchro} provides an estimate of the invariant set of the interconnected systems with quantized measurements.
To this purpose, we rely on a result of
exponential synchronization under static output
feedback control laws  and time-varying graphs
which has been investigated in
\cite{scardovi.sepulchre.aut}. In the following statement, we
recall Theorem 4 of \cite{scardovi.sepulchre.aut} specialized to
the case of time-invariant undirected graphs:
\begin{theorem}
Let Assumption \ref{a.passivity} hold and suppose that
the pair $(C,A)$ is observable.
Let the communication graph be
undirected and connected, and denote $z=(D^T\otimes I_p) w$ as in {\rm (\ref{z})} with $w=[w_1^T\ldots w_N^T]^T$.
Then the solutions of
\be\label{unquantized} \dot \xi_i =  A \xi_i -B \dst\sum_{k=1}^M
d_{ik}z_k,\;i=1,2,\ldots, N \ee satisfy
\be\label{synchro.unquant} \lim_{t\to+\infty} \left\| \xi_i(t)-
\frac{\mathbf{1}_N^T\otimes I_n}{N}\xi(t) \right\|=0
\ee
where $\xi=[\xi_1^T \xi_2^T \cdots \xi_N^T]^T$ and the
convergence is exponential. More precisely, the solutions converge
exponentially to the solution of $\dot \xi_0=A \xi_0$ initialized to
the average of the initial conditions of the systems
{\rm (\ref{unquantized})}, i.e.\
$\xi_0(0)=\mathbf{1}_N^T\otimes I_n \xi(0)/N$.
\end{theorem}

Let $\tilde \xi= \xi-\frac{\mathbf{1}_N\mathbf{1}_N^T\otimes I_n}{N}\xi=(\Pi \otimes I_n) \xi$, with $\Pi=I_N-
\frac{\mathbf{1}_N\mathbf{1}_N^T}{N}$, be the disagreement vector.
From (\ref{unquantized}), $\tilde \xi(t)$ obeys the equation
\be\label{A.tilde}
\ba{rcl} \dot{\tilde \xi} &=&
\underbrace{[I_N\otimes A -(I_N\otimes B)(DD^T\otimes
I_p)(I_N\otimes C)]}_{\tilde A}\tilde \xi \ea\ee
and the convergence
result (\ref{synchro.unquant}) can be restated as
$\lim_{t\to+\infty} ||\tilde \xi(t)||=0$.
%
%
The proof of the result rests on showing that the Lyapunov function
\[
V(\tilde \xi)= \tilde \xi^T (I_N\otimes P) \tilde \xi
\]
along the solutions of (\ref{A.tilde}) satisfies the inequality
\[
\dot V(\tilde \xi)\le -\lambda_2 ||(\Pi \otimes I_p) 
\tilde \xi
||^2,
\]
where $\lambda_2$ is the algebraic
connectivity of the graph, i.e.\ the smallest non-zero eigenvalue
of the Laplacian $L=DD^T$.
Then the thesis descends from the observability assumption and
Theorem 1.5.2 in \cite{sastry.bodson.book}.

%


%

\section{Quantized coordination control}\label{quant.coord.contr}

\subsection{A practical agreement problem}

Despite the generality allowed by the passivity approach of
\cite{arcak.tac07},  in this paper
we focus on an agreement problem. By
an agreement problem it is meant a special case of coordination in
which all the variables $x_i$ connected by a path converge to each
other. In the problem formulation in Section
\ref{sect.passivity.approach}, this amounts to have ${\cal
A}_k=\{\mathbf{0}\}$ for all $k=1,2,\ldots, M$. When using
{\em statically}\footnote{The use of dynamic quantizers can lead to asymptotic results. See
\cite{carli.bullo.sicon09, li.et.al.tac11} for a few results for discrete-time systems. Dealing with continuous-time 
systems and dynamic quantizers poses a few extra challenges, which are not addressed in this paper. }
quantized measurements, however, it is a well established fact
(\cite{kashyap.et.al.aut07,dimarogonas.johansson.aut10,
ceragioli.et.al.aut11}) that a coordination algorithm
leads
to a practical agreement result, meaning that each variable
$z_k$ converges to a compact set 
containing
the origin, rather than
to the origin itself. Motivated by this observation, we set in
this paper a weaker convergence goal, namely for each
$k=1,2,\ldots, M,$ we ask the target set ${\cal A}_k$ to be of the
form:
\be\label{A.k} {\cal A}_k=\bigcart_{j=1}^p [-a, a]
\ee
where $a$ is a
positive constant
and the symbol $\times$ denotes the Cartesian
product. Then the design procedure of Section
\ref{sect.passivity.approach} prescribes to choose a non-negative potential
function $P_k(z_k)$ which is radially unbounded on its domain of
definition and such that (\ref{properties.P}) holds.
If such a function exists then the control law is chosen via
(\ref{psi.via.P}). To take into account the presence of quantized
measurements,  the nonlinearities $\psi_k$ on the right-hand side of
(\ref{psi.via.P}) should take
the form
%
\be\label{quantized.psi.via.P} \psi_k(z_k)=
\chi_k(\qd(z_k))\;, \ee
%
with $\chi_k$  to be defined later.\\
The presence of quantized measurements, i.e.\ of
$\qd(z_k)$,
makes the right-hand side of (\ref{psi.via.P}) discontinuous
and asks for a redefinition of the requirements (\ref{properties.P}).
In this paper, we look for a {\em locally Lipschitz} radially unbounded non-negative
functions $P_k$ which satisfy
\be\label{properties.P.disc} P_k(z_k)=0\;\mbox{and}\; \mathbf{0}\in
\partial
P_k(z_k)\quad \mbox{if and only if}\quad  z_k\in {\cal
A}_k, \ee
where $\partial P_k(z_k)$ is the Clarke generalized gradient (see
Appendix \ref{a0} for a definition) which is needed since $P_k(z_k)$ is
now not continuously differentiable. Similarly to (\ref{properties.P}), we are asking
${\cal A}_k$ to be the set of all local and global minima for $P_k(z_k)$. \\
A candidate function $P_k(z_k)$ with the properties
(\ref{properties.P.disc}) and such that a function
$\chi_k$ exists for which (\ref{psi.via.P}), (\ref{quantized.psi.via.P}) hold, is the
function
%
%
\be\label{potential} P_k(z_k)= \dst\sum_{j=1}^p
\dst\int_0^{z_{kj}} \qd(s) ds, \ee
where $z_{kj}$ is the $j$th component of the vector $z_k\in \R^p$
(see Fig.~\ref{P_z} for a picture of $P_k(z_k)$).
\begin{figure}
\begin{center}
\includegraphics[scale=0.5]{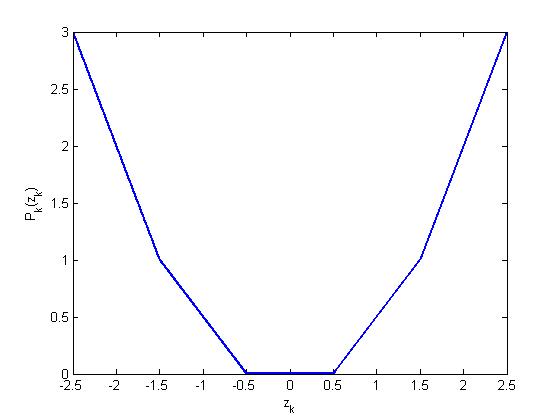}
\caption{\label{P_z}
The graph of  $P_k(z_k)$ with $z_k\in \R$ and $\Delta=1$. }
\end{center}
\end{figure}
%

Such a function is defined on all $\R^p$, is radially unbounded
and locally Lipschitz. By Rademacher's theorem
(\cite{clarke.et.al.book}, Chapter 3) 
it is
differentiable almost everywhere. In all the points of $\R^p$
where it is differentiable
$\nabla
P_k(z_k)=\qd(z_k)$
i.e.\
(\ref{psi.via.P}), (\ref{quantized.psi.via.P}) holds with
$\chi_k={\rm Id}$ (${\rm Id}:\R^p\to \R^p$ is the identity function). Bearing in mind
the definitions (\ref{scalar.quantizer}) and (\ref{A.k}), to
satisfy
(\ref{properties.P.disc})
it is
necessary and sufficient to set
$a=\frac{\Delta}{2}$.
In what follows we examine the evolution of the system
(\ref{closed.loop.x.xi}) under the control law:
\be\label{quantized.control} u_i=-\dst\sum_{k=1}^M d_{ik}
\qd(z_k)\;,\; i=1,2,\ldots,N.
\ee

\subsection{Closed-loop system}
Similarly to (\ref{arcak.control.compact}), we write the quantized
control law in compact form as:
\be\label{control.by.quantized.meas} u=-(D\otimes I_p)\qd(z)\;,
\ee
where
$\qd(z)=(\qd(z_1)^T\ldots
\qd(z_M)^T)^T$.
The closed-loop system then
takes the following expression:
\be\label{closed.loop.x.xi2}\ba{rcl}
\dot x&=& h(\xi)+\mathbf{1}_N\otimes v\\
\dot \xi&=&
 f(\xi)+g(\xi)(-(D\otimes I_p)\qd(z)),
 \ea\ee
where $z=(D^T\otimes I_p) x$ and the maps  $f,g,h$ are as in (\ref{closed.loop.x.xi}).

\subsubsection{Control scenario and implementation}
Before proceeding to the analysis of the system, it is important to motivate in more detail the control scenario we consider and how the overall control scheme is implemented.
\\[0.1cm]
For each pair of neighboring agents, one of the two is equipped with a sensor which  continuously take the relative measurement with respect to its neighbor, e.g.,  a sonar or a radar. Not all the agents 
are equipped with these sensors since they might have very dedicated tasks  in the formation and space must be saved for other hardware needed to accomplish these tasks. On the other hand, since these agents need information to maintain their positions in the formation, they receive  such information in quantized form from their neighbors via a digital communication channel.
\\[0.1cm]
The implementation of the control law in (\ref{closed.loop.x.xi2})
can be given by the {\it quantization-based 
distributed control protocol} as follows.\\[0.1cm]
({\it Initialization}.) At time $t_0=0$, all sensors measures
$z_k(t_0)$, $k=1,2,\ldots,p$. The processing units collocated with
the sensors computes $\qd(z_k(t_0))$ and the resulting value is
broadcasted to its neighbors. Each agent $i$ computes the local
control law $u_i$ as in \rfb{quantized.control} and the control
value is held until new information is available. Note that the
closed-loop system   evolves according to
\[
\ba{rcl}
\dot x&=& h(\xi)+\mathbf{1}_N\otimes v\\
\dot \xi&=&
 f(\xi)+g(\xi)(-(D\otimes I_p)\qd(z(t_0)))
 \ea\]
for all $t>t_0$ until new information is available.\\[0.1cm]
({\it Quantization-based 
transmission and control update}.) Let
$\ell=1$ and let $t_\ell$ be the smallest time at which
$t_\ell>t_{\ell-1}$ and the processing unit of a sensor in the
$k$-th edge detects that $\qd(z_k(t_\ell))\neq
\qd(z_k(t_{\ell-1}))$. In this case, the quantized information
$\qd(z_k(t_\ell))$ is transmitted to its neighbor and the local
control law of the $i$-th and $j$-th agents, where $(i,j)$ is the
pair of nodes linked by the $k$-th edge, is updated by
\[
\ba{rcl}
u_i(t_\ell)&=& u_i(t_{\ell-1}) -d_{ik}\big(\qd(z_k(t_\ell))-\qd(z_k(t_{\ell-1}))\big)\\
u_j(t_\ell)&=& u_j(t_{\ell-1})
-d_{jk}\big(\qd(z_k(t_\ell))-\qd(z_k(t_{\ell-1}))\big).
 \ea
\]
There is no other information exchange and hence the rest of the
agents maintains their local control values. The local control law
is now fixed until new information is transmitted again. 
For all $t>t_\ell$ and until this new transmission occurs, the evolution of the
closed-loop system 
is given by
\[
\ba{rcl}
\dot x&=& h(\xi)+\mathbf{1}_N\otimes v\\
\dot \xi&=&
 f(\xi)+g(\xi)(-(D\otimes I_p)\qd(z(t_\ell))).
 \ea\]
The quantization-based event-triggered control update process is
iterated with the index value $\ell$ incremented by one.\\[0.2cm]
A few remarks are in order: 
\begin{description}
\item{(i)}
The construction outlined above results in a sequence of unevenly-spaced sampling times $t_\ell$, $\ell\in \N$, at which sensors located at the systems  broadcast quantized information to neighboring systems.
This information is used by local controllers to update the control value. The control laws turn out to be  piece-wise constant functions of time whose value is updated whenever new information is received. 
\item{(ii)}
Notice that even the agent which measures $z_k$ implements a control
law in which $\qd(z_k)$ is used instead of $z_k$ itself.  This is mainly motivated by our
need to preserve a ``symmetric" structure in the closed-loop system.
In fact,  given agents $i,j$ and their relative distance $z_k$, in
the case of unquantized information, agent $i$ would use $z_k$ in
the control law, and agent $j$, $-z_k$. Similarly, in the case of
quantized measurements, it is very helpful in the analysis to employ
$\qd(z_k)$ in the control law for agent $i$ and $-\qd(z_k)$ in the
one for agent $j$.
\item{(iii)}
In the quantization-based transmission and control protocol described
above, the solution of the closed-loop system is not prevented to
evolve along a discontinuity surface.
In practice, due to delays in the transmission and in the
implementation of the control law, this could result in chattering, which
is of course undesirable in the present context, since it would
require fast information transmission. Nevertheless, in
\cite{ceragioli.et.al.aut11} a new class of hybrid quantizers have
been considered which prevent the occurrence of chattering. This
class could also be used  for the problem at hand in this paper, but
this is not pursued further for the sake of brevity. 
\end{description}

\subsubsection{A notion  of solution }
The system (\ref{closed.loop.x.xi2}) has a discontinuous right-hand side due to the
presence of the quantization functions and its analysis requires
a suitable notion of solution. In this paper we adopt
Krasowskii solutions. In fact, it was shown in
\cite{ceragioli.et.al.aut11} that Carath\'eodory solutions may not
exist for agreement problems. Moreover, Krasowskii solutions
include Carath\'eodory solutions and the results we derive for the
former also holds for the latter
in case they exist.\\
Denoted by $\dot X(t)=F(t,X)$ the system
(\ref{closed.loop.x.xi2}), a function $X(\cdot)$ defined on an
interval $I\subset \R$ is a Krasowskii solution to the system on
$I$ if it is  absolutely continuous and satisfies the differential
inclusion (\cite{OH:79})
\begin{equation}\label{diffinclusion}
\dot X(t)\in {\cal{K}}(F(t,X)):=\bigintersect_{\delta >0}{\overline
{\rm co}}\, (F(t,B(X,\delta)))
\end{equation}
for almost every (a.e.) $t\in I$. The operator ${\rm co}(S)$ denotes the convex closure of
$S$, i.e. the smallest closed set containing the convex hull of $S$.
Since the right-hand side of
(\ref{closed.loop.x.xi2}) is
locally bounded, local existence of Krasowskii solutions is
guaranteed (\cite{OH:79}). \\
The differential inclusion corresponding to the system
(\ref{closed.loop.x.xi2}) can be written explicitly.
%
More
precisely, for every $k\in\{1,2,\ldots, M\}$ and
$i\in\{1,2,\ldots,p\}$, 
we observe that ${\cal{K}}\qd(r)$ is given by
\begin{equation*}
{\cal{K}}\qd(r) =
\left\{\begin{array}{ll}m\Delta
& r \in \big((m-\frac{1}{2})\Delta,(m+\frac{1}{2})\Delta\big), m\in\mathbb{Z} \\
\left[m\Delta, (m+1)\Delta \right] & r =
(m+\frac{1}{2})\Delta, m\in \mathbb{Z}. \end{array} \right.
\end{equation*}
Using ${\cal{K}}\qd$,
the differential inclusion
(\ref{diffinclusion}) for (\ref{closed.loop.x.xi2}) can be written
as
\be\label{closed.loop.x.xi2.diffincl}\ba{rcl}
\dot x&=& h(\xi)+\mathbf{1}_N\otimes v\\
\dot \xi&\in&
 f(\xi)+g(\xi)(-(D\otimes I_p){\cal K}\qd(z)),
 \ea\ee
where
${\cal K}\qd(z):=\bigcart_{k=1}^M {\cal K}\qd(z_k)$,
${\cal K}\qd(z_k):=\bigcart_{j=1}^p {\cal K}\qd(z_{kj})$.
Note that we have used the calculus rule for the set-valued map
${\cal K}$
\[
{\cal{K}}\big[g(\xi)(-(D\otimes I_p)\qd(z))\big] =
g(\xi)(-(D\otimes I_p){\cal K}\qd(z))
\]
(see also \cite[Theorem 1]{cortes.csm08, paden.sastry.cs87}).
The Krasowskii solutions
to (\ref{closed.loop.x.xi2}) are
also Filippov solutions as it  follows from \cite[Lemma
2.8]{OH:79} for a piecewise continuous vector field $F$. Since
every Carath\'eodory solutions to (\ref{closed.loop.x.xi2}) is
also a Krasowskii solution to (\ref{closed.loop.x.xi2}), the
stability properties of the Krasowskii solutions are also
inherited by the classical Carath\'eodory solutions \cite{OH:79}
in case the latter exist.

\subsubsection{Analysis}

Recalling that
\[
(D^T\otimes I_p)(\mathbf{1}_N\otimes v)=\mathbf{0}
\]
and bearing in mind (\ref{feedf}),
the system (\ref{closed.loop.x.xi2}) in the coordinates $(z,\xi)$
writes as
\be\label{closed.loop.z.xi}\ba{rcl}
\dot z &=& (D^T\otimes I_p) h(\xi)\\
\dot \xi&=& f(\xi)+g(\xi)(-(D\otimes I_p)\qd(z)). \ea\ee
Even the system above is discontinuous and again its solutions must be intended in the Krasowskii sense.
It is straightforward to verify that, given any Krasowskii solution $(x,\xi)$ to
(\ref{closed.loop.x.xi2}), the function $(z,\xi)=((D^T\otimes I_p)
x, \xi)$ is a Krasowskii solution to (\ref{closed.loop.z.xi}).
The differential inclusion corresponding to (\ref{closed.loop.z.xi})
is easily understood from (\ref{closed.loop.x.xi2.diffincl}).
In what follows we investigate the asymptotic properties of the
Krasowskii solutions to (\ref{closed.loop.z.xi}) and
infer stability properties of (\ref{closed.loop.x.xi2}).\\
A few notions of
nonsmooth control theory which are used in the proofs are recalled
in the Appendix \ref{a0}. The first fact we notice is the
following:
\begin{lemma}\label{lemma.main}
Let Assumption \ref{assum.passivity} hold
and let the communication graph $G$ be undirected and connected. Then
any Krasowskii solution to (\ref{closed.loop.z.xi}) converges to
the set of Krasowskii equilibria: {\rm
\be\label{set} \{(z, \xi)\,:\, \xi={\mathbf 0},\; {\mathbf 0}\in
(D\otimes I_p){\cal K}\qd(z)\}\;. \ee }
\end{lemma}

\begin{proof}
To analyze the system (\ref{closed.loop.z.xi}) we consider the
Lyapunov function
\[\ba{rcl}
V(z,\xi)&=& \underbrace{\dst\sum_{i=1}^N S_i(\xi_i)}_{S(\xi)}+
\underbrace{\dst\sum_{k=1}^M P_k(z_k)}_{P(z)}\\
&=& \sum_{i=1}^N S_i(\xi_i)+ \dst\sum_{k=1}^M \dst\sum_{j=1}^p
\dst\int_0^{z_{kj}} \qd(s) ds\;. \ea\]
%
The function
is a locally Lipschitz and regular function. In fact, each term
$\int_0^{z_{kj}} \qd(s) ds$
is convex and as such it
is regular (\cite[Proposition 2.3.6]{clarke.book},\cite{cortes.csm08}). Then the sums $P_k(z_k)$ and $\sum_{k=1}^M P_k(z_k)$
are also regular.
The function $V(z,\xi)$ is nonnegative and
vanishes on the set of points such that $\xi=\mathbf{0}$ and
$z_{kj}\in [-a, a ]$
for all $k\in \{1,2,\ldots,M\}$
and
all $j\in \{1,2,\ldots,p\}$. \\
In order to apply the LaSalle's invariance principle for the differential inclusions as given in Theorem \ref{diffLaSalletheorem} in Appendix \ref{a0},
we analyze the set-valued derivative $V$ with respect to (\ref{closed.loop.z.xi}) as follows. Define
\begin{equation*}
\dot{\overline V}(z,\xi)=\{a\in \R\,:\,  \exists w\in {\cal K}\tilde F(z,\xi)\;{\rm s.t.}\;\;a=\langle p,w\rangle, \; \;
 \forall p\in \partial V(z,\xi) \;\}\;.
\end{equation*}
where
$\langle \cdot,\cdot\rangle$ denotes the standard inner product
and
$\tilde F(z,\xi)$ the right-hand side of (\ref{closed.loop.z.xi}).
We first observe that by definition of $V(z,\xi)$, $\partial
V(z,\xi)$ and $\partial P(z)$, $p\in \partial V(z,\xi)$ implies
the existence of $p_z\in \partial P(z)$ such that
\[
p=\left(\ba{c} p_z
\\
\nabla S(\xi)
\ea\right).
\]
Moreover,
if $w\in {\cal K}\tilde F(z,\xi)$ then there exists $w_z\in {\cal K}\qd(z)$
(\cite{OH:79},\cite{paden.sastry.cs87}) such that
\[
w= \left(\ba{c} (D^T\otimes I_p) h(\xi)\\ f(\xi) \ea\right) +
\left(\ba{c} \mathbf{0}\\ g(\xi) \ea\right)(-D\otimes I_p)w_z.
\]
Let now $p\in \partial V(z,\xi)$ and $w\in {\cal K}\tilde
F(z,\xi)$  and write
\be\label{p.w}\ba{rcl}
\langle p,w\rangle &=& \langle\nabla S(\xi), f(\xi)+g(\xi)(-D\otimes I_p)w_z\rangle+\\[2mm]
&&
\hspace{3cm}\langle p_z,(D^T\otimes I_p) h(\xi)\rangle\\
&\le& - \dst\sum_{i=1}^N W_i(\xi_i)+\langle h(\xi),(-D\otimes I_p)w_z\rangle+\\[2mm]
&& \hspace{3cm}\langle p_z,(D^T\otimes I_p) h(\xi)\rangle, \ea\ee
where the inequality is a consequence of (\ref{passivity}).
Suppose now that for some $(z,\xi)$, $\dot{\overline V}(z,\xi)\ne
\emptyset$. Then, for every $a\in \dot{\overline V}(z,\xi)$
and for every $p\in\partial V(z,\xi)$, there exists $w\in {\cal K}\tilde
F(z,\xi)$ such that $a=\langle
p,w\rangle$. By definition of $\qd(z)$ and $P(z)$,
$\partial P(z)={\cal K}\qd(z)$
(\cite{OH:79},\cite{paden.sastry.cs87}). Then $a=\langle
p,w\rangle$ holds  in particular when
\[
p=\left(\ba{c} p_z
\\ \nabla S(\xi)
\ea\right)=\left(\ba{c} w_z
\\ \nabla S(\xi)
\ea\right).
\]
with $w_z\in {\cal K}\qd(z)$. Then (\ref{p.w}) becomes:
\[\ba{rcl}
\langle p,w\rangle &\le&  - \dst\sum_{i=1}^N W_i(\xi_i)+\langle h(\xi),(-D\otimes I_p)w_z\rangle+\\
&& \langle w_z,(D^T\otimes I_p) h(\xi)\rangle=- \dst\sum_{i=1}^N
W_i(\xi_i). \ea\] Hence, for all $(z,\xi)$ such that
$\dot{\overline V}(z,\xi)\ne \emptyset$, we have that
\be\label{sv.der} \dot{\overline V}(z,\xi)=\{a\in \R\,:\, a\le -
\dst\sum_{i=1}^N W_i(\xi_i)\}\;. \ee
Since $\frac{d}{dt} V(z(t),\xi(t))\in \dot{\overline
V}(\xi(t),z(t))\subseteq (-\infty, 0]$ for almost every $t$,
$V(z(t),\xi(t))$ cannot increase, and  any Krasowskii solution
$(z(t),\xi(t))$ is bounded. Hence,
$(z(t),\xi(t))$ exists for all $t$.\\
%
Given any initial condition $(z(0),\xi(0))$, the set ${\cal S}$ such
that $V(z,\xi)\le V(z(0),$ $\xi(0))$ is a strongly  invariant set
for (\ref{closed.loop.z.xi}) which contains the initial condition.
An application of Theorem \ref{diffLaSalletheorem} in Appendix \ref{a0} shows that any Krasowskii
solution converges to the largest weakly invariant set contained in
${\cal S}\cap \{(z,\xi)\,:\,\mathbf{0} \in \dot{\overline V}(z,\xi)\}$.
\\
Moreover, in view of (\ref{sv.der}), the set $Z$ of points
$(z,\xi)$ such that $0\in \dot{\overline V}(z,\xi)$ is
contained in the set of points such that $\xi=\mathbf{0}$. Hence,
any point of  the largest weakly invariant set contained in
${\cal S}\cap Z$
is such that $\xi=\mathbf{0}$. Pick a point
$(z,\mathbf{0})$
on this invariant set. Then in order for  a Krasowskii solution to
(\ref{closed.loop.z.xi}) starting from this point to remain in the
invariant set, it must be true that $\mathbf{0}\in
f(\mathbf{0})+g(\mathbf{0})(-(D\otimes
I_p){\cal K}\qd(z))=g(\mathbf{0})(-(D\otimes I_p){\cal K}\qd(z))$. Since the
matrix
$g(\mathbf{0})$ is full-column-rank  (recall that each $g_i(\mathbf{0})$ is full-column rank),
the inclusion above
requires the existence of $w_z\in {\cal K}\qd(z)$ such that
$(D\otimes I_p)w_z=\mathbf{0}$.
In other words, the largest weakly invariant set included in
${\cal S}\cap Z$
is contained in the set (\ref{set}). Finally,
observe that, taken any point in the set (\ref{set})
as initial condition for (\ref{closed.loop.z.xi}), at least
a Krasowskii solution $(z(t),\xi(t))$ originating from this point must coincide with the
trivial solution, i.e.\ $(z(t),\xi(t))=(\mathbf{0}, \mathbf{0})$ for all $t$.
Hence, any point in (\ref{set}) is a Krasowskii equilibrium for
(\ref{closed.loop.z.xi}).
%
\end{proof}

It is now possible to prove the following:
\begin{theorem}\label{theo.main}
Let Assumption \ref{assum.passivity} hold
and let the communication graph $G$ be undirected and connected.
Let $v:\R_{+}\to \R^p$
be a bounded and piecewise continuous
function and
$\Delta$
be a positive number.
Then any Krasowskii solution to
(\ref{closed.loop.x.xi2}) converges to the set
\be\label{target.set}
\{(x,\xi)
: \xi=\mathbf{0}\;,\;z\in ({\cal A}_1\times\ldots\times {\cal
A}_M), z=(D^T\otimes I_p)x\}, \ee where the sets ${\cal A}_k$'s
are defined in (\ref{A.k}), with
$a=\Delta/2$.
Moreover, $\lim_{t\to+\infty} [\dot x(t)-\mathbf{1}_N\otimes
v(t)]=\mathbf{0}$.
\end{theorem}

\begin{proof}
Consider any Krasowskii solution $(x(t),\xi(t))$ to
(\ref{closed.loop.x.xi2}), whose existence is guaranteed locally.
It can also be extended for all $t\in [0, +\infty)$. In fact
suppose by contradiction this is not true, i.e.\ $(x(t),\xi(t))$
is defined on the interval $[0, t_f)$, with $t_f<+\infty$.
Define $(z(t),\xi(t))=((D^T\otimes I_p) x(t), \xi(t))$.
$(z(t),\xi(t))$ is a
Krasowskii solution to (\ref{closed.loop.z.xi}).
As proven before,
such a solution is bounded on its its domain of definition. Since
by (\ref{closed.loop.x.xi2}) $\dot x(t)=h(\xi(t))+v(t)$ and both
the terms on the right-hand side are bounded, then $x(t)$ grows
linearly in $t$ and therefore it must be bounded on the maximal
interval of definition, i.e.\ $t_f=+\infty$. Hence both
$(x(t),\xi(t))$ and $(z(t),\xi(t))=((D^T\otimes I_p) x(t),
\xi(t))$ are defined for all $t$. Moreover, by Lemma
\ref{lemma.main}, $z(t)=(D^T\otimes I_p) x(t)$ converges to the
set of points (\ref{set}), i.e.\ to
\be\label{set.x} \{(x,\xi)\,:\, \xi=\mathbf{0},\; \mathbf{0}\in
(D\otimes I_p){\cal K}\qd(z),\; z=(D^T\otimes I_p)x\}\;. \ee
%
Let $(x,\mathbf{0})$
belong to the set (\ref{set.x}). Then
$z=(D^T\otimes I_p)x$, i.e. $z$ belongs to the span of $D^T\otimes
I_p$ and there exists $w_z\in {\cal K}\qd(z)$ such that
$(D\otimes I_p)w_z=\mathbf{0}$. The two conditions  imply that
$\langle w_z,z\rangle =0$. We claim that then necessarily $z\in
{\cal A}_1\times\ldots\times {\cal A}_M$, with the sets ${\cal
A}_i$'s given in (\ref{A.k}). In fact, if this is not true, then
there must exist a pair of indices $j,k$ such that
$|z_{kj}|>a$.
This implies that the entry $k+j$ of the vector
$w_z$ is different from zero and also $w_{z,k+j}\cdot z_{kj}>0$.
Moreover, since $w_z\in {\cal K}\qd(z)$, for any pair of indices
$i, \ell$ such that $i\ne k$ or $\ell \ne j$, $w_{z, i+\ell}\cdot
z_{i\ell}\ge 0$. This contradicts that $\langle w_z,z\rangle=0$.
Then we have proven that the set (\ref{set.x}) is included in the
set
\be\label{set.x.new} \{(x,\xi)\,:\, \xi=\mathbf{0},\; z\in {\cal
A}_1\times\ldots\times {\cal A}_M,\; z=(D^T\otimes I_p)x\}\;. \ee
%
Hence, any Krasowskii  solution $( x(t),\xi(t))$ to
(\ref{closed.loop.x.xi2}) converges to
a subset of (\ref{set.x.new}).\\
As for the second part of the statement, any Krasowskii solution
to (\ref{closed.loop.x.xi2}) is such that $\dot
x(t)-\mathbf{1}_N\otimes v(t)=h(\xi(t))$, and since we have proven
that $\xi(t)\to \mathbf{0}$ as $t\to \infty$, we have also proven
that  $\lim_{t\to+\infty} [\dot x(t)-\mathbf{1}_N\otimes
v(t)]=\mathbf{0}$.
\end{proof}

\subsection{Examples}\label{ss.examples}
We provide three examples of application of the quantized agreement
result described above.

{\it Agreement of single integrators by quantized measurements.} We
specialize
 the proof of Theorem \ref{theo.main}  to the agreement problem for single integrators.
 This problem stems from the case when in (\ref{closed.loop.x.xi2}) the mapping from $u=-(D\otimes I_p)\qd(z)$
 to $y=h(\xi)$ is an identity operator.
The closed-loop system (\ref{closed.loop.x.xi2}) reduces to:
\be\label{single.integrator} \ba{rcl}
\dot x&=& -(D\otimes I_p) \qd(z)\\
z&=&(D^T\otimes I_p) x \ea\ee 
which using the variables $z$
becomes 
\be\label{single.integrator.z} \dot z=-(D^T\otimes
I_p)(D\otimes I_p) \qd(z)=-(D^T D\otimes I_p)\qd(z) \ee 
We analyze
this system using the function  $P(z)$ introduced above. 
Since $\partial P(z)={\cal K}\qd(z)$, then for all $z$ such that
$\dot{\overline{P}}(z)\ne \emptyset$, we have
\[
\dot{\overline{P}}(z)=\{a\in \R: \exists w_z\in {\cal K}\qd(z) \;{\rm s.t.}\;a=-||(D\otimes I_p)w_z||^2\}.
\]
Hence, arguments as in Lemma \ref{lemma.main} give that all
the Krasowskii solutions to (\ref{single.integrator.z}) converge
to the set of points
$\{z\,:\, \mathbf{0}\in (D\otimes I_p){\cal K}\qd(z)\}$.
On the other hand, by Theorem \ref{theo.main}, any
Krasowskii solution  $x(t)$ to (\ref{single.integrator}) is such
that $z(t)=(D^T\otimes I_p) x(t)$ converges to
$\{x\,:\,\mathbf{0}\in (D\otimes I_p){\cal K}\qd(z),\; z=(D^T\otimes I_p)
x\}$
which is included in the set $\{z\,:\, z\in {\cal
A}_1\times\ldots\times {\cal A}_p,\; z=(D^T\otimes I_p) x\}$.\\ 
Let
$x$ be any Krasowskii solution to (\ref{single.integrator}) with
$z=(D^T\otimes I_p)x$. Take any two variables $x_i, x_j$ whose
agents are connected by the edge $k$. Consider for the sake of
simplicity that each quantizer has the same parameter $\Delta$.
%
%
Then  $z_k=x_i-x_j$ converges asymptotically to a square of the
origin whose edge is not longer than $\Delta$. If the agents are
not connected by an edge but by a path, then each entry of
$x_i-x_j$ is in magnitude bounded by $\Delta \cdot d$, with $d$
the diameter of the graph.
The result can be compared with Theorem 4 in
\cite{dimarogonas.johansson.aut10}. One difference is that, while
trees are considered in \cite{dimarogonas.johansson.aut10},
connected graphs are considered here. Moreover, in
\cite{dimarogonas.johansson.aut10} the scalar states are
guaranteed to converge to a ball of radius $\frac{||D^T D||
\sqrt{M}}{\lambda_{min}(D^T D)}\Delta$. Hence, denoted by $\rho$
the ratio $\frac{||D^T D||}{\lambda_{min}(D^T D)}$  and considered
the bound $M\le N-1$, any two states $x_i, x_j$ may differ for
$2\rho \Delta \sqrt{N-1}$.
The passivity approach considered here
yields that they differ for not more than $d\cdot\Delta$, where
$d$ grows as $O(\rho \log(N))$ (\cite{chung.ejc06}) for not
complete and regular graphs (graphs with all the nodes having the
same degree), 
thus leading to a smaller region of convergence, the quantizer resolution $\Delta$ being the same. 
%
%

{\it Agreement of double integrators by quantized measurements}
Consider the case of $N$ agents modeled as
\be\label{double.integrator} \ddot x_i=f_i\;,\quad i=1,2,\ldots,N,
\ee with $x_i, f_i\in \R^2$, for which we want to solve the
agreement problem with quantized measurements. This means that all
the agents should practically converge towards the same position
and also asymptotically evolve with the same velocity $v$. The preliminary
feedback (\cite{arcak.tac07}) \be\label{preliminary.feedback}
f_i=-K_i(\dot x_i-v)+\dot v+u_i\;,\quad K_i=K_i^T, \ee with $u_i$
to design, and the change of variables $\xi_i=\dot x_i-v$,
 makes the closed-loop system
\[\ba{rcl}
\dot x_i &=& \xi_i+v\\
\dot \xi_i &=& -K_i \xi_i+u_i\\
y_i &=& \xi_i \ea\] passive with storage function
$S_i(\xi_i)=\frac{1}{2} \xi_i^T \xi_i$ and $W_i(\xi_i)=-K_i
\xi_i^T \xi_i$. The system above is in the form
(\ref{feedb.path}). Theorem \ref{theo.main} guarantees that the
Krasowskii solutions of (\ref{double.integrator}),
(\ref{preliminary.feedback}), (\ref{control.by.quantized.meas})
%
%
converges asymptotically to the set (\ref{target.set}) and that
all the agents' velocities converge to $v$.
In other words, the
formation achieves practical position agreement and convergence to
the prescribed velocity.
%
%
\begin{remark}[Consensus for double integrators with
velocity feedback]\label{rem.double}
A different but related consensus problem consists of
designing local controllers in such a way that each double integrator
converges asymptotically to the same position and velocity. In this case,
no external reference velocity is provided and
the velocity to which all the systems converge is the average of the initial
velocities (\cite{tanner.et.al.cdc03.I}). The controller which guarantees this coordination task uses
both position and velocity feedback (observe that the communication
graphs for the position measurements and for the velocity measurements
can be different). It then makes sense to consider
the problem in the presence of quantized relative position and velocity
measurements.
This has been investigated in \cite{liu.cao.ifac11}.
%
\end{remark}


{\em The case of unknown reference velocity.} If the reference velocity $v$ is not available to all the
agents, then \cite{bai.et.al.scl08, bai.et.al.aut09} suggest to
replace it with an estimate which is generated by each agent on
the basis of the current available measurements. Here we examine
this control scheme when the measurements are quantized. We
consider the special case in which the unknown reference velocity
is constant. Then each agent $i$, with
the exception of one which acts as a leader and can access the
prescribed reference velocity $v$, use an estimated version of
$v$, namely $\hat v_i$ has to be generated on-line starting from the
available local measurements. The agent's dynamics
(\ref{feedb.path}) becomes
\be\label{feedb.path.est} \ba{rcll}
\dot x_i &=& y_i+\hat v_i\\
\dot \xi_i &=& f_i(\xi_i)+g_i(\xi_i) u_i\\
y_i &=& h_i(\xi_i), & i=1,2,\ldots, N, \ea\ee
with $\hat v_i=v$ if $i=1$ (without loss of generality agent
$1$ is taken as the leader), and otherwise  generated by
\[
\dot{\hat v}_i = \Lambda_i  u_i
\]
with $\Lambda_i=\Lambda_i^T>0$ and $u_i$ as in
(\ref{arcak.control}). Observe that in this case, the
estimated velocity is updated via quantized measurements.
Consider the closed-loop system
\be\label{closed.loop.x.xi.theta}
\ba{rcl}
\dot x&=& h(\xi) +
\mathbf{1}_N\otimes \hat
v
\\
\dot \xi&=& f(\xi) -
g(\xi)(D\otimes I_p)\qd(z)
\\
\dot{\hat v} &=& -\Lambda (D\otimes I_p)\qd(z)
\ea\ee
where $\Lambda={\rm diag}(\Lambda_1,\ldots,\Lambda_N)$ and
$z=(D^T\otimes I_p) x$.
Let
\[
\hat v_i(t)=v+\hat v_i(t)-v= v+(\hat v_i(t)-v)=:v+\tilde v_i(t)
\]
where $\tilde v_1=\mathbf{0}$.
Rewrite the system using the
coordinates $z$ and $\tilde \theta$ and obtain
\be\label{closed.loop.x.xi.tilde.theta} \ba{rcl} \dot z&=&
(D^T\otimes I_p)[h(\xi) + \tilde
v]
\\
\dot \xi&=& f(\xi) -
g(\xi)(D\otimes I_p)\qd(z)\\
\dot{\tilde v } &=& \Lambda
(D\otimes I_p)\qd(z) \ea\ee
where in the second equation it was exploited again the fact that $(D^T\otimes
I_p)\mathbf{1}_N=\mathbf{0}$.
\\
One can now proceed as in the proof of Lemma \ref{lemma.main}.
Consider the Lyapunov function
\[\ba{rcl}
V(z,\xi,\tilde v )&=&
S(\xi)+P(z)+\dst\frac{1}{2}\tilde v ^T\Lambda^{-1}\tilde v
\ea\]
and let $\tilde F(z,\xi,\tilde v )$ be the right-hand
side of (\ref{closed.loop.x.xi.tilde.theta}). For any $p\in
\partial V(z,\xi,\tilde v )$ and $w\in {\cal K}\tilde
F(z,\xi,\tilde v )$  consider
\be\label{p.w2}\ba{rcl}
\langle p,w\rangle &=& \langle\nabla S(\xi), f(\xi)+g(\xi)(-D\otimes I_p)w_z\rangle+\\[2mm]
&&
\langle p_z,(D^T\otimes I_p) [h(\xi)+\tilde v ]\rangle
+\\[2mm]
&&\langle \Lambda^{-1}\tilde v , -\Lambda (D\otimes I_p)w_z\rangle\\
\ea\ee
where
\[
p=\left(\ba{c} p_z
\\ \nabla S(\xi)\\
\Lambda^{-1}\tilde v  \ea\right)
\]
and
\[
w= \left(\ba{c} (D^T\otimes I_p)[h(\xi)+\tilde v ]\\
 f(\xi)\\ \mathbf{0}
\ea\right) + \left(\ba{c}
 \mathbf{0}\\ g(\xi)\\
-\Lambda \ea\right)(-D\otimes I_p)w_z.
\]
As in Lemma \ref{lemma.main} one proves that
$\langle p,w\rangle \le - \sum_{i=1}^N W_i(\xi_i)$
and therefore that
\be\label{sv.der2} \dot{\overline V}(z,\xi,\tilde v )=\{a\in
\R\,:\, a\le - \dst\sum_{i=1}^N W_i(\xi_i)\}\;. \ee
Hence, any Krasowskii solution $( z(t),\xi(t), \tilde v(t))$
is bounded and exists for all $t$. Let
${\cal S}$
be the level set such that
$V(z,\xi, \tilde v)\le V(z(0),\xi(0), \tilde v(0))$ and $Z$
the set of points $( z,\xi,\tilde v)$ such that $\mathbf{0}\in
\dot{\overline V}(z,\xi,\tilde v )$.
%
%
Then any solution $(z,\xi,\tilde v )$
converges to the largest weakly invariant subset contained in
${\cal S}\cap Z$.
Observe that $Z\subset \{( z,\xi,\tilde v):\xi=0\}$. Moreover, for a set in
${\cal S}\cap Z$
to be weakly invariant, it must be true that
$\mathbf{0}\in {\cal K}\tilde F(z,\xi,\tilde v )$ with $\tilde F(z,\xi,\tilde v )$
the right-hand
side of (\ref{closed.loop.x.xi.tilde.theta}). These two facts together
imply that there must exist
$w_z\in {\cal K}\qd(z)$
such that $(D\otimes I_p)w_z=\mathbf{0}$
and additionally $(D^T\otimes I_p)\tilde v=\mathbf{0}$. The
latter implies that $\tilde v=(\mathbf{1}_N\otimes I_p)c$ for some $c\in
\R$. Since $\tilde v_1=\mathbf{0}$, then on the largest weakly
invariant set contained in $S\cap Z$ it is also true that $\tilde v =\mathbf{0}$.
Hence it
follows that any Krasowskii solution to
(\ref{closed.loop.x.xi.tilde.theta}) converges to the set
\be\label{set.theta} \{( z, \xi,\tilde v)\,:\, \xi={\mathbf 0},\; {\mathbf
0}\in (D\otimes I_p){\cal K}\qd(z),\; \tilde v ={\mathbf
0}\} \;. \ee
Note that each point in the set is a  Krasowskii equilibria of
(\ref{closed.loop.x.xi.tilde.theta}).\\
One can then focus on the system (\ref{closed.loop.x.xi.theta})
and follow the same arguments of Theorem \ref{theo.main} to conclude
that the solutions of the closed-loop system converge to the set
where all the systems evolve with the same velocity, achieve
practical consensus on the position variable and the estimated velocities
$\hat v_i$ converge to the true reference velocity $v$.
\begin{proposition}
Let Assumption \ref{assum.passivity} hold and let the communication graph $G$ be undirected and connected.
Let $v\in \R^p$ be a constant vector and
$\Delta$ a positive number.
Then any Krasowskii solution to
(\ref{closed.loop.x.xi.theta}) converges to the set
\be
\{(x,\xi,\hat v)
: \xi=\mathbf{0},z\in ({\cal A}_1\times\ldots\times {\cal
A}_M), z=(D^T\otimes I_p)x, \hat v=\mathbf{1}_N\otimes v\}, \ee
where the sets ${\cal A}_k$'s
are defined in (\ref{A.k}), with
$a=\Delta/2$.
In particular, $\lim_{t\to+\infty} [\dot x(t)-\mathbf{1}_N\otimes
v]=\mathbf{0}$.

\end{proposition}

\begin{remark}{\bf (Velocity error
feedback)}
Instead of the control law (\ref{arcak.control.compact})
$u=-(D\otimes I_p)\psi(z)$, the control law  proposed in
\cite{bai.et.al.aut09} considers an additional velocity error
injection (namely,  $\sum_{j\in {\cal
N}(i)}(\dot x_j-\dot x_i)$, with ${\cal N}(i)$ the set of
neighbors with respect to which the agent $i$ can measure the
relative velocity). This modified control law guarantees velocity tracking (with
time-varying reference velocity) and agreement of the variables
$x$ without relying on the convergence of the estimated velocity
to the actual value. However,
the use of this additional
velocity feedback term in the presence of quantization poses
a few additional challenges which are not tackled in this paper.
See also Remark \ref{rem.double} for more comments in this respect.
\end{remark}

\section{Quantized synchronization of passive systems} \label{quant.synchro}


We turn now our attention to the systems in (\ref{lti}) where the
control law that we consider is a static quantized output-feedback
control law of the form
\be\label{control} u=-(D\otimes
I_p)\qd(z)\;\mbox{with}\; z = (D^T \otimes I_p)w. \ee The overall
closed-loop
system is 
\be\label{cl}
\ba{rcl}
\dot \xi &=& (I_N\otimes A)\xi-(I_N\otimes B)(D\otimes I_p)\qd(z)\\
z &=& (D^T\otimes I_p)w=(D^T\otimes I_p)(I_N\otimes C)\xi.
\ea\ee
Applications where synchronization problems under communication
constraints and passivity are relevant are reviewed in \cite{fradkov.et.al.cs1.09}. Later in
this section, we briefly discuss another example where the use
of  quantized measurements for synchronization can be useful.

To study the robustness of the synchronization algorithm to
quantized measurements we need a more explicit characterization of
the exponential stability of (\ref{A.tilde}). To this purpose we
introduce a different Lyapunov function which is characterized in the
following lemma. As we consider time-invariant graphs, observability can be
replaced by a detectability assumption.
\begin{lemma}\label{firstlemma}
Let $(C,A)$ be detectable and $\Pi=I_N-
\frac{\mathbf{1}_N\mathbf{1}_N^T}{N}$. The integral
\[
R:=\dst\int_{0}^{+\infty} (\Pi \otimes I_n)^T {\rm e}^{\tilde A^T
s} {\rm e}^{\tilde A s} (\Pi \otimes I_n) ds,
\]
with $\tilde A$ as in (\ref{A.tilde}),
is finite
and satisfies
\be\label{R.bound}
||R||\le \dst\int_{0}^{+\infty}
\left\|\left(
\ba{ccc}
\exp(A-\lambda_2 BC)s & \ldots & \mathbf{0}_{n\times n}\\
\vdots & \ddots & \vdots \\
\mathbf{0}_{n\times n} & \ldots & \exp(A-\lambda_N BC)s
\ea\right)
\right\|^2
ds.
\ee
Moreover, the Lyapunov function
\[
U(\tilde \xi)=\tilde \xi^T R \tilde \xi
\]
satisfies the following:
\be\label{c1c2}
\ba{c}
c_1||\tilde \xi||^2\le U(\tilde \xi)\le c_2||\tilde \xi||^2\\[2mm]
\nabla U(\tilde \xi)\cdot\tilde A\tilde \xi\le -
||\tilde \xi||^2
\ea\ee
for each $\tilde \xi\in \R^{nN}$.
\end{lemma}

\begin{proof}
The proof is given in the Appendix \ref{a1}.
\end{proof}

The first fact we prove about (\ref{cl}) is that the control law (\ref{control}) achieves practical
synchronization of the outputs:
\begin{proposition}
Let Assumption \ref{a.passivity} hold
and let the communication graph $G$ be undirected and connected. Then
any Krasowskii solution to (\ref{cl}) converges to the largest weakly invariant
subset contained in
\be\label{lwis}
\{\xi\in R^{nN}: |z_{kj}|\le \frac{\Delta}{2},\forall\;k=1,2,\ldots,M,\; j=1,2,\ldots p\},
\ee
with $z=(D^T\otimes I_p)(I_N\otimes C)\xi$.
\end{proposition}
\begin{proof}
Any Krasowskii solution to (\ref{cl}) satisfies the differential inclusion
\[
\dot \xi \in (I_N\otimes A)\xi-(I_N\otimes B)(D\otimes I_p){\cal K}\qd(z).
\]
Consider the Lyapunov function $V(\xi)=\xi^T (I_N\otimes P) \xi$. Then,
for any $\xi\in R^{Nn}$ and any
$\nu\in {\cal K}\qd(z)$,
with
$z=(D^T\otimes I_p)(I_N\otimes C)\xi$,
we have
\[\ba{rcl}
\dot V(\xi)&:=& \nabla V(\xi)\cdot [(I_N\otimes A)\xi-(I_N\otimes B)(D\otimes I_p)\nu]\\
&=& 2 \xi^T (I_N\otimes PA) \xi- 2 \xi^T (I_N\otimes PB)(D\otimes I_p)\nu.
\ea
\]
Using Assumption \ref{a.passivity} and the definition of $z$ we further obtain
that for all $\nu\in {\cal K}\qd(z)$,
\be\label{V.dot}\ba{rcl}
\dot V(\xi)
&\le& - 2 \xi^T (I_N\otimes C^T)(D\otimes I_p)\nu \\[2mm]
&=& - 2 z^T \nu \leq 0.
\ea
\ee
This shows that $V(\xi(t))$ cannot increase and that $\xi(t)$ is bounded. Moreover, by  LaSalle's
invariance principle for differential inclusion (Appendix \ref{a0}, Theorem \ref{diffLaSalletheorem}),
any Krasowskii solution
converges to the largest weakly invariant subset contained in
\[
\{\xi\in \R^{Nn}: \exists\; \nu\in {\cal K}\qd(z)\,{\rm s.t.}\,
\nabla V(\xi)\cdot [(I_N\otimes A)\xi-(I_N\otimes B)(D\otimes I_p)\nu]=0\}.
\]
In view of (\ref{V.dot}), any point $\xi$ in this set is such that
$z_{kj}\, \nu_{kj}=0$ for all $k=1,2,\ldots, M$ and for all $j=1,2,\ldots,p$.
Since 
$\nu_{kj}\in  {\cal K}\qd(z_{kj})$, then
$z_{kj}\, \nu_{kj}=0$ implies that $|z_{kj}|\le \frac{\Delta}{2}$. This ends the proof.
\end{proof}

\begin{remark}{\bf (Practical output synchronization)}
Similarly to
 the consensus problem under quantized measurements
(see Section \ref{ss.examples}), a consequence of the previous statement is that
any two outputs  $w_i,w_j$
practically asymptotically synchronize. Namely, considered any Krasowskii solution $\xi(t)$ and
the corresponding output $w(t)=(I_N\otimes C)\xi(t)$, for each $\ell=1,2, \ldots, n$
and each $t\ge 0$, the difference
$|w_{i\ell}(t)-w_{j\ell}(t)|$ is upper bounded by a quantity which
asymptotically converges to $d\frac{\Delta}{2}$, with $d$ the diameter
of the graph.
\end{remark}

The proof of the proposition above clearly does not rely on the linearity of the systems
but rather on the passivity property. Hence, if one considers nonlinear passive systems,
that is systems for which a positive definite continuously differentiable storage function
$V_i(\xi_i)$ exists such that $\nabla V_i\cdot f_i(\xi_i,u_i)\le w_i^T\cdot u_i$, with $w_i=h_i(\xi_i)$,
then for the closed-loop system $\dot \xi_i=f_i(\xi_i, u_i)$, with $u_i$ given in (\ref{control}),
and $i=1,2,\ldots, N$, it is still true that the overall storage function $V(\xi)=\sum_{i=1}^N V_i(\xi_i)$ satisfies
the inequality $\dot V(\xi)\le - 2z^T \nu$ for all $z=(D^T\otimes I_p)h(\xi)$ and all
$\nu\in {\cal K}\qd(z)$.
Hence, the following holds:
\begin{proposition}
Let Assumption \ref{assum.passivity} hold
and let the communication graph $G$ be undirected and connected. Then
any Krasowskii solution to the systems (\ref{feedb.path}) in closed-loop with {\rm $u=-(D\otimes I_p) \qd(z)$}  and
$z=(D^T\otimes I_p)h(\xi)$ converges to the largest weakly invariant subset contained in the set (\ref{lwis}).
\end{proposition}


The next lemma states a property of the average of the  solutions to (\ref{cl})
which helps to better characterize the region  where the solutions converge.
\begin{lemma}\label{l.invariance}
Let Assumption \ref{a.passivity} hold
and let the communication graph $G$ be undirected and connected.
Any Krasowskii solution $\xi(t)$ to (\ref{cl}) satisfies
\[
(\mathbf{1}_N^T\otimes I_n) \xi(t)={\rm e}^{At} (\mathbf{1}_N^T\otimes I_n) \xi(0)
\]
for all $t\ge 0$.
\end{lemma}
\begin{proof}
Observe that for almost every $t$:
\[\ba{rcl}
\dst\frac{d}{dt}
(\mathbf{1}_N^T\otimes I_n)\xi(t)
&=&
(\mathbf{1}_N^T\otimes I_n)\dst\frac{d}{dt}\xi(t)\\
&\in & (\mathbf{1}_N^T\otimes I_n)(I_N\otimes A)\xi(t)-(\mathbf{1}_N^T\otimes I_n)(I_N\otimes B)(D\otimes I_p) {\cal K}\qd(z)
\ea\]
Bearing in mind that for matrices $F\in \R^{m\times n}$ and $G\in \R^{p\times q}$, the following property of
the Kronecker product holds:
\[
F \otimes G= (F\otimes I_p)(I_n\otimes G)=(I_m\otimes G)(F\otimes I_q),
\]
one can further show that
\be\label{Star}\ba{rcl}
\dst\frac{d}{dt}
(\mathbf{1}_N^T\otimes I_n)\xi(t)
&\in &(\mathbf{1}_N^T\otimes I_n)(I_N\otimes A)\xi(t)-B(\mathbf{1}_N^T D\otimes I_p) {\cal K}\qd(z)
\\
&= &(\mathbf{1}_N^T\otimes I_n)(I_N\otimes A)\xi(t)\\
&= &A(\mathbf{1}_N^T\otimes I_n) \xi(t)\\
\ea\ee
where in the equality before the last one it was exploited the fact that $\mathbf{1}_N^T D=\mathbf{0}_M^T$,
which holds by definition of the incidence matrix $D$. Hence, any Krasowskii solution
$\xi(t)$ is such that the average $(\mathbf{1}_N^T\otimes I_n) \xi(t)$ satisfies
\[
(\mathbf{1}_N^T\otimes I_n) \xi(t)={\rm e}^{At} (\mathbf{1}_N^T\otimes I_n) \xi(0).
\]
\end{proof}

The following result provides an estimate of the region where the solutions converge
and
%
shows
%
practical synchronization
under quantized relative measurements:
\begin{theorem}\label{th.syncro}
Let Assumption \ref{a.passivity} hold
and let the communication graph $G$ be undirected and connected. Assume that $(C,A)$ is detectable. Then
for any Krasowskii solution $\xi(t)$ to
\be\label{synchro.quantized}
\ba{rcl}
\dot \xi &=& (I_N\otimes A)\xi-(I_N\otimes B)(D\otimes I_p){\tt q}
((D^T\otimes I_p)(I_N\otimes C)\xi)
%
\ea\ee
 there exists a finite time $T$ such that $\xi(t)$
satisfies
\be\label{bound}
\dst\frac{1}{\sqrt{pM}}
\left\| \xi(t)-(\mathbf{1}_N\otimes
I_n)\frac{(\mathbf{1}_N^T\otimes
I_n)\xi(t)}{N}\right\|
\le 2\sqrt{\dst\frac{c_2}{c_1}}||R||\, ||B||\,
||D\otimes I_p|| \Delta
\ee
for all $t\ge T$,
where $c_1, c_2, ||R||$ are defined in (\ref{R.bound}), (\ref{c1c2}).
Moreover, $\frac{\mathbf{1}_N^T\otimes I_n}{N}\xi(t) = \xi_0(t)$ where $\xi_0(t)$ is the solution of
$\dot \xi_0(t) = A\xi_0(t)$ with the initial condition $\xi_0(0)=\frac{\mathbf{1}_N^T\otimes I_n}{N}\xi(0)$.
\end{theorem}

\begin{proof}
By definition, any Krasowskii solution $\xi$ to
(\ref{synchro.quantized}) is such that $\tilde \xi=(\Pi\otimes
I_n)\xi$, with $\Pi=I_N-
\frac{\mathbf{1}_N\mathbf{1}_N^T}{N}$,  satisfies
\[
\dot {\tilde \xi} \in  (I_N\otimes A)\tilde \xi-(I_N\otimes B)(D\otimes I_p)
{\cal K}\qd((D^T\otimes I_p)(I_N\otimes C)\xi), 
\]
where similar manipulations as in (\ref{Star}) were used.
Moreover, any $\nu\in {\cal K}\qd((D^T\otimes
I_p)(I_N\otimes C)\xi)$ is such that $||\nu-(D^T\otimes
I_p)(I_N\otimes C)\xi||\le \dst\sqrt{pM}\frac{\Delta}{2}$.
Under the assumption on the detectability of $(C,A)$, we
can consider the Lyapunov function $U(\tilde \xi)$ introduced in
Lemma \ref{firstlemma}. For any $\xi$ and any $\nu\in {\cal K}\qd((D^T\otimes I_p)(I_N\otimes
C)\xi)$,
\[\ba{rcl}
&&
\nabla U(\tilde \xi)
[(I_N\otimes A)\tilde \xi-(I_N\otimes B)(D\otimes I_p)\nu]\\[3mm]
&=& 
\nabla U(\tilde \xi)
[(I_N\otimes
A)-(I_N\otimes B)(DD^T\otimes I_p)
(I_N\otimes C)]\tilde \xi +\\[3mm]
&&
\nabla U(\tilde \xi)
(I_N\otimes B)(D\otimes
I_p)[(D^T\otimes I_p)
(I_N\otimes C)\tilde \xi - \nu]\\[3mm]
&\le & -||\tilde \xi||(||\tilde \xi||- ||R||\, ||B||\, ||D\otimes
I_p||\dst\sqrt{pM}\Delta).
\ea\]
Hence, for $||\tilde \xi||>
\frac{1}{2}||R||\, ||B||\, ||D\otimes I_p||\dst\sqrt{pM}\Delta$,
\[
\nabla U(\tilde \xi)
[(I_N\otimes
A)\xi-(I_N\otimes B)(D\otimes I_p)\nu]\le- \dst\frac{||\tilde
\xi||^2}{2}\le -\dst\frac{1}{2 c_2} U(\tilde \xi).
\]
It follows that any Krasowskii solution converges in finite time
to the set of points $\tilde \xi $ such that
\[
||\tilde \xi|| \le 2\sqrt{\dst\frac{c_2}{c_1}}||R||\, ||B||\,
||D\otimes I_p||\dst\sqrt{pM}\Delta
\]
from which the thesis is proven by definition of $\tilde \xi$.

The proof of the final claim follows from the
fact that by Lemma \ref{l.invariance},  for all $t\ge 0$,
\[
\tilde \xi(t)=\xi(t)- (\mathbf{1}_N\otimes
I_n)\frac{(\mathbf{1}_N^T\otimes
I_n)\xi(t)}{N}=\xi(t)- (\mathbf{1}_N\otimes
I_n)\frac{\dst{\rm e}^{At}
(\mathbf{1}_N^T\otimes
I_n)\xi(0)}{N}.
\]
%
%
%
\end{proof}

\begin{remark}{\bf (Role of $||R||$)}
In the case $A=0$, $B=C=1$,
the bound on $R$ reduces to
$||R||\le \frac{1}{2\lambda_2}$, where $\lambda_2$ is the algebraic connectivity of the graph.
In this case, the size of the region of convergence in (\ref{bound})
resembles the estimate given in Theorem 1 and Corollary 1
in \cite{ceragioli.et.al.aut11} for quantized consensus of single
integrators. Theorem \ref{th.syncro} can  be viewed
as the extension of the results in  \cite{ceragioli.et.al.aut11} to the problem of
synchronization of  linear multi-variable
passive systems by quantized output feedback.
\end{remark}

\subsection{Examples}

In the following examples, we discuss how  synchronization with quantized
measurements can play a role in a decentralized output regulation problem in
which heterogeneous systems asymptotically agree on the trajectory to track.

{\em Output synchronization for heterogeneous linear systems.} 
In \cite{wieland.et.al.aut11} (see also \cite{bai.et.al.book}, Section 3.6) the following
problem is investigated. Given $N$ heterogeneous linear systems
\be\label{examplesynchro}\ba{rcll}
\dot x_i &=& F_i x_i +G_i u_i\\
y_i &=& H_i x_i, & i=1,2,\ldots, N
\ea\ee
%
with $(F_i, G_i)$ stabilizable and $(H_i, F_i)$ detectable,
and a graph $G$
(which here, as usual in this paper, we assume static
undirected and connected), find a feedback control law $u_i$
for each system $i$ (i) which uses relative measurements concerning only the systems which are connected
to the system $i$ via the graph $G$ and (ii) such that output synchronization
is achieved, i.e.\ $\lim_{t\to \infty} || y_i(t)-y_j(t)||=0$ for all $i,j \in \{1,2,\ldots, N\}$.\\
Excluding the trivial case in which the closed-loop system has an attractive set of equilibria where the
outputs are all zero, the authors of \cite{wieland.et.al.aut11} show that
the output synchronization problem
for $N$ heterogeneous systems is solvable if and only if there exist matrices $S,R$
such that $\lim_{t\to \infty} ||y_i(t)-R{e}^{-St}w_0||=0$ for each $i \in \{1,2,\ldots, N\}$, for some $w_0$.
Moreover, provided that $\sigma (S)\subset j\R$, the controllers which
solve the regulation problem are
\be\label{synchro.regulator}
\ba{rcl}
\dot{\hat x}_i &=& F_i \hat x_i +G_i u_i + L_i (\hat y_i- C_i x_i)\\
\hat y_i &=& H_i \hat x_i\\
u_i &=& K_i(\hat x_i- \Pi_i \xi_i) +\Gamma_i \xi_i
\ea
\ee
where $\xi_i\in \R^p$ are the exosystem states that synchronize via communication channels and are described by
\be\label{synchro.exosystem}
\ba{rcl}
\dot \xi &=& (I_N\otimes S)\xi-(I_N\otimes B)(D\otimes I_p)z\\
z &=& (D^T\otimes I_p)(I_N\otimes C)\xi,
\ea\ee
where
$D$ is the incidence matrix associated to the graph,
%
%
the pair $(C,S)$
is detectable the matrices $L_i, K_i$ are such that $F_i+G_iK_i, F_i+L_iH_i$ are Hurwitz, and $\Pi_i, \Gamma_i$ are matrices
which solve the regulator equations
\[\ba{l}
F_i\Pi_i+G_i\Gamma_i=\Pi_i S\\
H_i\Pi_i=R.
\ea\]
The controllers (\ref{synchro.regulator})--(\ref{synchro.exosystem}) are a modified form of the ones in \cite[Eq. (10)]{wieland.et.al.aut11} where
in the latter, the local controller communicates the entire exosystem state $\xi_i$ to its connecting nodes.
When the relative measurement $z_k$ is transmitted via a digital communication line, then this information is quantized and
the variable $z$ in the controller (\ref{synchro.regulator})--(\ref{synchro.exosystem}) is replaced by its quantized form $\qd(z)$. \\
Let the eigenvalues of $S$ have in addition multiplicity of one in the minimal polynomial,
so that we can restrict $S$ to be skew-symmetric without loss of generality and $B=C^T$. Then
the exosystems
\be\label{exo.lti}\ba{rcll}
\dot \xi_i &=&  S \xi_i + B u_i\\
w_i &=& C \xi_i & i=1,2,\ldots, N
\ea\ee
trivially satisfy Assumption \ref{a.passivity}. 
Then Theorem \ref{th.syncro} applies and the solutions $\xi_i$, $i=1,2,\ldots, N$, of (\ref{synchro.exosystem}) practically synchronize under the quantization of $z$.
It is then possible to see that the closed-loop system of (\ref{examplesynchro}) and the controllers (\ref{synchro.regulator})--(\ref{synchro.exosystem}) with
$z$ replaced by $\qd(z)$ achieves practical output synchronization.
This follows from similar arguments as in \cite[Theorem 5]{wieland.et.al.aut11} where \cite[Theorem 1]{wieland.et.al.aut11}, which is used in the proof the theorem,
is replaced by Theorem \ref{th.syncro}.

Before ending the section, we remark that Theorem \ref{th.syncro} also holds under a
slightly different set of conditions which do not require passivity.
\begin{assumption}\label{positive.realness}
Let $(A,B,C)$ be stabilizable and detectable, and assume that
\be\label{spr}
[I_p+\lambda_N {\mathbf{G}}][I_p+\lambda_2\mathbf{G}]^{-1}
\ee
is strictly positive real where $\mathbf{G}(s)= C(sI-A)^{-1}B$ is the transfer function of {\rm (\ref{lti})} and $\lambda_N$ is the largest eigenvalue of $L$.
\end{assumption}
Under Assumption \ref{positive.realness}, the results in Theorem \ref{th.syncro} still hold mutatis mutandis.
Indeed, by the multivariable circle criterion
in \cite[Theorem 3.4]{jayawardhana2009}, $(A-\lambda_i BC)$ is Hurwitz for every non-zero eigenvalue
$\lambda_i$ of $L$.
This implies that (\ref{A.tilde}) is exponentially stable
(this is evident from 
the proof of Lemma \ref{firstlemma} -- see the Appendix) and Lemma \ref{firstlemma}
and \ref{l.invariance} continue to hold. As a consequence the proof of Theorem \ref{th.syncro} holds
word by word under the assumption that $(A,B,C)$ is minimal and
Assumption \ref{positive.realness} holds.

{\em The case of output synchronization with filtered and quantized signals.}
As a concrete example to the case of exosystems satisfying Assumption \ref{positive.realness}, we consider again the closed-loop systems in the previous example where the heterogenous linear systems (\ref{examplesynchro}) are interconnected with the controllers (\ref{synchro.regulator})--(\ref{synchro.exosystem}) with
\begin{equation}
S = \left[\begin{array}{ccc} 0 & \omega & 0 \\ -\omega & 0 & 0 \\ 0 & a & -a
\end{array}\right], B = \left[\begin{array}{c} 0 \\ 1 \\ 0
\end{array}\right], C = \left[\begin{array}{ccc} 0 & 0 & 1
\end{array}\right].
\end{equation}
The system $(S,B,C)$ can be considered as a cascade interconnection of
a second-order oscillator with frequency $\omega$ and a low-pass filter with a cut-off frequency $a$, and its transfer function is given by
\[
G(s) = \frac{as}{(s^2+\omega^2)(s+a)}.
\]
Using the above $(S,B,C)$, the interconnected exosystems (\ref{synchro.exosystem}) with quantized
measurement $\qd(z)$ resemble a network of oscillators where the relative measurements $z_k$
are filtered and quantized.
In the limiting case $a\to\infty$, the exosystems are given by (\ref{exo.lti}) where
\begin{equation}
A = \left[\begin{array}{cc} 0 & \omega\\ -\omega & 0
\end{array}\right], B = \left[\begin{array}{c} 0 \\ 1
\end{array}\right], C = \left[\begin{array}{cc} 0 & 1
\end{array}\right];
\end{equation}
and it satisfies Assumption \ref{a.passivity}. A direct application of Theorem \ref{th.syncro}
shows that (\ref{bound}) holds with
\begin{equation*}
\|R\| \leq \int_0^\infty{\left\|
\left(
\begin{array}{ccc} \exp\left(\begin{array}{cc}0 & \omega \\ -\omega & -\lambda_2\end{array}\right)s & \cdots &
\mathbf{0}_{n\times n}\\ \vdots & \ddots & \vdots \\ \mathbf{0}_{n\times n} & \cdots & \exp\left(\begin{array}{cc}0 &
\omega \\ -\omega & -\lambda_N\end{array}\right)s \end{array}\right)\right\| \dd s}.
\end{equation*}
In particular, if $\lambda_2>4\omega^2$, then $\|R\| \leq \frac{1}{\lambda_2-\sqrt{\lambda_2^2-4\omega^2}}$. \\
On the other hand, if $0<a<\infty$, i.e., when the low-pass filter is used, then it can be checked that
\begin{align*}
& \inf_{\nu}\text{Re}\left(\frac{1+\lambda_N G(i\nu)}{1+\lambda_2G(i\nu)}\right) \geq 0 \\
\Leftrightarrow & \inf_{\nu}(a\omega^2-a\nu^2)^2+((\omega^2+\lambda_Na)\nu-\nu^3)((\omega^2+\lambda_2a)\nu-\nu^3) \geq 0.
\end{align*}
Note that for a sufficiently large $a>0$, the above condition holds. Thus, the cut-off frequency $a$ can be designed based only on the knowledge of $\lambda_2, \lambda_N$ and $\omega$,
such that the exosystems (\ref{exo.lti}) satisfy Assumption \ref{positive.realness}.\\
In both cases, practical output synchronization of the closed-loop systems (\ref{examplesynchro})--(\ref{synchro.exosystem}) with quantized $\qd(z)$ is obtained.

\section{Conclusions}\label{sec.final}
The passivity approach to coordinated control problems presents
several interesting features such as for instance the possibility
to deal with agents which have complex and high-dimensional
dynamics. In this paper we have shown how it also lends itself to
take into account the presence of quantized measurements. Using
the passivity framework along with appropriate tools from
nonsmooth control theory and differential inclusions, we have
shown that many of the results of \cite{arcak.tac07, scardovi.sepulchre.aut} continue to
hold in an appropriate sense in the presence of quantized
information. We believe that the results presented in the paper
are a promising addition to the existing literature on
continuous-time consensus and coordinated control under
quantization (\cite{dimarogonas.johansson.aut10,
ceragioli.et.al.aut11,
liu.cao.ifac11,frasca.arxiv11}).\\
Many additional aspects deserve attention in future work on the
topic. The approach to quantized coordinated control pursued in
this paper appears to be suitable to tackle more complex formation
control problems such as those considered e.g.\ in Section II.C of
\cite{arcak.tac07}, \cite{dimarogonas.johansson.aut10}, Section 4
and \cite{tanner.et.al.cdc03.I}. These possible extensions can
also benefit from the results of \cite{bai.et.al.aut09}.\\
In the paper it was not discussed whether or not the use of
quantized measurements yields sliding modes. Sliding modes were
shown to occur in problems of quantized consensus for single
integrators (\cite{ceragioli.et.al.aut11}) and hysteretic
quantizers were introduced to overcome the problem. A similar
device could prove useful in quantized coordination problems.\\
The literature on synchronization and coordination problems which
exploit passivity is rich (see e.g.\
\cite{pogromsky.nijmeijer.csi01,scardovi.sepulchre.aut,
chopra.spong.cdc06, schaft.maschke.necsys10} and references
therein) and the problems presented there could be reconsidered in
the presence of quantized measurements. The book \cite{bai.et.al.book}
provides many other results of cooperative control within the passivity
approach. These results are all potentially extendible to the case
in which quantized measurements are in use.

\bigskip

\noindent {\bf Acknowledgement} The authors would like to thank Paolo
Frasca for a remark on the first example in Section 3.3.

\bibliographystyle{plain}
\bibliography{QUICK-BIBLIO_16Aug11}

\appendix

\section{Notation}\label{a-1}
The Kronecker product of the matrices $A\in \R^{m\times n}$, $B\in
\R^{p\times q}$ is the matrix
\[
A\otimes B=\left(\ba{ccc}
a_{11}\, B & \ldots & a_{1n}\, B\\
\vdots & \ddots & \vdots\\
a_{m1}\, B & \ldots & a_{mn}\, B \ea\right).
\]
See e.g.\ \cite{arcak.tac07,scardovi.sepulchre.aut} for some basic
properties.

\section{Nonsmooth control theory tools}\label{a0}
A few tools of nonsmooth control theory which are used throughout
the paper are recalled in this appendix (see
\cite{bacciotti.ceragioli.esaim99,cortes.csm08}  for more details).
Consider the differential inclusion \be\label{di} \dot x\in F(x),
\ee with $F:\R^n \to 2^{\R^n}$ a set-valued map. We assume for $F$
the standard assumptions for which existence of solutions is
guaranteed (\cite{OH:79}). $x_0\in \R^n$ is a Krasowskii
equilibrium for (\ref{di}) if the function $x(t)=x_0$ is a
Krasowskii solution to (\ref{di}) starting from the initial
condition $x_0$, namely if $\mathbf{0}\in F(x_0)$. A set ${\cal
S}$ is weakly (strongly) invariant for (\ref{di}) if for any
initial condition $\overline x\in {\cal S}$ at least one (all the)
Krasowskii solution $x(t)$ starting from $\overline x$ belongs
(belong) to  ${\cal S}$ for all $t$ in the domain of definition of
$x(t)$. Let $V:\R^n\to \R$ be a locally Lipschitz function. Then
by Rademacher's theorem the gradient of $V$ exists almost
everywhere. Let $N$ be the set of measure zero where $\nabla V(x)$
does not exist. Then the Clarke generalized gradient of $V$ at $x$
is the set $\partial V(x)= {\rm co}\{\lim_{i\to+\infty} \nabla
V(x_i)\,:\, x_i\to x,\; x_i\not\in S\,,\, x_i\not\in N\}$ where
$S$ is any set of measure zero  in $\R^n$. We define the
set-valued derivative of $V$ at $x$ with respect to (\ref{di}) the
set $\dot{\overline V}(x)=\{a\in \R:\, \exists v\in {\cal
K}f(x)\;{\rm s.t.}\;a=p\cdot v,\;\forall p\in \partial V(x)\}$.
The definition of regular functions used in the following
nonsmooth LaSalle invariance principle can be found e.g.\ in
\cite{bacciotti.ceragioli.esaim99}:
\begin{theorem}\label{diffLaSalletheorem}
{\rm (\cite{bacciotti.ceragioli.esaim99, cortes.aut06})} Let
$V:\R^n\to \R$ be a locally Lipschitz and regular function. Let
$\overline x\in {\cal S}$, with ${\cal S}$ compact and strongly
invariant for (\ref{di}). Assume that for all $x\in {\cal S}$
either $\dot{\overline V}(x)=\emptyset$ or $\dot{\overline
V}(x)\subseteq (-\infty, 0]$. Then any Krasowskii solution to
(\ref{di}) starting from $\overline x$ converges to the largest
weakly invariant subset contained in ${\cal S}\cap \{x\in
\R^n\,:\, \mathbf{0}\in \dot{\overline V}(x)\}$, with $\mathbf{0}$
the null vector in $\R^n$.
\end{theorem}

\section{Proof of Lemma \ref{firstlemma}}\label{a1}

\begin{proof}
Following \cite{olfati-saber.murray.tac04}, Theorem 3 (see also \cite{fax.murray.tac04}), we introduce the
$N\times N$ nonsingular matrices
\footnote{
The matrices $T, T^{-1}$ transform the Laplacian matrix $L=DD^T$
into its diagonal form. The columns of $T$ form an orthonormal basis
of $\R^N$.
}
\[
T=(\mathbf{1}_N/\sqrt{N}\;v_2\ldots v_N),\;
T^{-1}=\left(\ba{cccc}
\mathbf{1}_N/\sqrt{N}& w_2 & \ldots &w_N \ea\right)^T
\]
and notice the following:
\be\label{exp.tilde.a}
\ba{l}
{\rm e}^{\tilde A s} (\Pi \otimes I_n)=(T \otimes I_n) \cdot\\[2mm]
\left(\ba{c|ccc}
\mathbf{0}_{n\times n} & \mathbf{0}_{n\times n} & \ldots & \mathbf{0}_{n\times n}\\
\hline
\mathbf{0}_{n\times n} & \exp(A-\lambda_2 BC)s & \ldots & \mathbf{0}_{n\times n}\\
\vdots & \vdots & \ddots & \vdots \\
\mathbf{0}_{n\times n} & \mathbf{0}_{n\times n} & \ldots & \exp(A-\lambda_N BC)s\\
\ea
\right)\cdot\\[10mm]
(T^{-1} \otimes I_n)(\Pi \otimes I_n)
\ea\ee
where
$0<\lambda_2<\ldots<\lambda_N$ are the non-zero eigenvalues of
$DD^T$. Since $(A,B,C)$ is passive and $(C,A)$ is detectable,  then
the matrices $A-\lambda_i BC$, $i=2,\ldots, N$ are Hurwitz. This
implies that  only exponentially stable modes are present in ${\rm
e}^{\tilde A s} (\Pi \otimes I_n)$, and therefore the integral
which defines $R$ exists and is finite.\\
Using the transformation matrix $T$, a routine computation shows that
\begin{align*}
& (\Pi \otimes I_n)^T {\rm e}^{\tilde A^T
s} {\rm e}^{\tilde A s} (\Pi \otimes I_n) \\ & = (\Pi T \otimes I_n)^T \left(\ba{c|ccc}
\mathbf{0}_{n\times n} & \mathbf{0}_{n\times n} & \ldots & \mathbf{0}_{n\times n}\\
\hline
\mathbf{0}_{n\times n} & \exp(A-\lambda_2 BC)^Ts & \ldots & \mathbf{0}_{n\times n}\\
\vdots & \vdots & \ddots & \vdots \\
\mathbf{0}_{n\times n} & \mathbf{0}_{n\times n} & \ldots & \exp(A-\lambda_N BC)^Ts\\
\ea
\right)(T^{-1} \otimes I_n)\cdot 
\end{align*}

\begin{align*}
& (T \otimes I_n) \left(\ba{c|ccc}
\mathbf{0}_{n\times n} & \mathbf{0}_{n\times n} & \ldots & \mathbf{0}_{n\times n}\\
\hline
\mathbf{0}_{n\times n} & \exp(A-\lambda_2 BC)s & \ldots & \mathbf{0}_{n\times n}\\
\vdots & \vdots & \ddots & \vdots \\
\mathbf{0}_{n\times n} & \mathbf{0}_{n\times n} & \ldots & \exp(A-\lambda_N BC)s\\
\ea
\right) (T^{-1}\Pi \otimes I_n)\\
& = ([v_2\ldots v_N] \otimes I_n)^T \left(\ba{ccc}
\exp(A-\lambda_2 BC)^Ts & \ldots & \mathbf{0}_{n\times n}\\
\vdots & \ddots & \vdots \\
\mathbf{0}_{n\times n} & \ldots & \exp(A-\lambda_N BC)^Ts\\
\ea
\right) \cdot \\ & \left(\ba{ccc}
\exp(A-\lambda_2 BC)s & \ldots & \mathbf{0}_{n\times n}\\
\vdots & \ddots & \vdots \\
\mathbf{0}_{n\times n} & \ldots & \exp(A-\lambda_N BC)s\\
\ea
\right) \left(\left[\ba{c}
w_2^T\\ \vdots\\ w_N^T \ea\right] \otimes I_n\right).
\end{align*}
Taking the norm of the matrix,
\[
||(\Pi \otimes I_n)^T {\rm e}^{\tilde A^T
s} {\rm e}^{\tilde A s} (\Pi \otimes I_n)|| \le
 \left\|\left(\ba{ccc}
\exp(A-\lambda_2 BC)s & \ldots & \mathbf{0}_{n\times n}\\
\vdots & \ddots & \vdots \\
\mathbf{0}_{n\times n} & \ldots & \exp(A-\lambda_N BC)s
\ea\right)\right\|^2
\]
from which (\ref{R.bound}) follows.
\\
Rewrite the function $U(\tilde \xi)$ as
\[\ba{rl}
&\dst\int_{t}^{+\infty} \tilde \xi^T(\Pi \otimes I_n)^T {\rm
e}^{\tilde A^T (\tau-t)}
{\rm e}^{\tilde A (\tau-t)} (\Pi \otimes I_n)\tilde \xi d\tau\\
=&\dst\int_{t}^{+\infty} ||\tilde \xi(\tau;\tilde \xi, t)||^2 d\tau
\ea\] where $\tilde \xi(\tau;\tilde \xi, t)$ is the solution to
(\ref{A.tilde}) at time $\tau$ starting from the initial condition
$\tilde \xi$ at time $t$. Following standard converse Lyapunov
theorem arguments (see e.g.\ Khalil, Theorem 4.12) one
easily proves that
\[
\ba{c} c_1||\tilde \xi||^2\le U(\tilde \xi)\le c_2||\tilde \xi||^2 \ea\]
Moreover,
\[\ba{rcl}
\nabla U(\tilde \xi)
\tilde A\tilde \xi &=&
\left[\tilde \xi^T(\Pi \otimes I_n)^T {\rm e}^{\tilde A^T s} {\rm
e}^{\tilde A s} (\Pi \otimes I_n)\tilde \xi
\right]_{s=0}^{s=+\infty}\\
&=&-||\tilde \xi||^2. \ea\]

\end{proof}

\end{document}